\documentclass[11 pt]{amsart}
\usepackage{amscd,amsfonts,amssymb,amsmath}
\usepackage{hyperref}
\usepackage{epsfig}
\usepackage{mathtools}
\usepackage{tabu}
\usepackage{tikz-cd}
\usepackage{blindtext}
\usepackage{ulem}
\usepackage{multicol}

\usepackage{enumitem}
\usepackage{physics}
\newtheorem{theorem}{Theorem}[section]
\newtheorem{corollary}[theorem]{Corollary}
\newtheorem{lemma}[theorem]{Lemma}
\newtheorem{proposition}[theorem]{Proposition}
\theoremstyle{definition}
\newtheorem{conjecture}[theorem]{Conjecture}

\newtheorem{remark}[theorem]{Remark}
\numberwithin{equation}{subsection}
\newtheorem*{ack}{Acknowledgement}
\newtheorem*{thm}{Theorem}
\usepackage[all,cmtip]{xy}

\usepackage{graphicx} 

\newcommand{\Aut}{\operatorname{Aut}}

\newcommand{\Inn}{\operatorname{Inn}}
\newcommand{\Out}{\operatorname{Out}}

\newcommand{\C}{\operatorname{C}}
\newcommand{\M}{\operatorname{M}}

\newcommand{\Z}{\operatorname{Z}}

\newcommand{\id}{\mathrm{id}}

\begin{document}

\title{Virtual planar braid groups and permutations}

\author{Tushar Kanta Naik}
\address{School of Mathematical Sciences, National Institute of Science Education and Research, Bhubaneswar, HBNI, P.O. Jatni, Khurda, Odisha 752050, India.}
\email{tushar@niser.ac.in}

\author{Neha Nanda}
\address{Department of Mathematical Sciences, Indian Institute of Science Education and Research (IISER) Bhopal, Bhopal Bypass Road, Bhauri, Bhopal 462066, Madhya Pradesh, India.}
\email{nehananda94@gmail.com}
\author{Mahender Singh}
\address{Department of Mathematical Sciences, Indian Institute of Science Education and Research (IISER) Mohali, Sector 81,  S. A. S. Nagar, P. O. Manauli, Punjab 140306, India.}
\email{mahender@iisermohali.ac.in}


\begin{abstract}
Twin groups and virtual twin groups are planar analogues of braid groups and virtual braid groups, respectively. These groups play the role of braid groups in the Alexander-Markov correspondence for the theory of stable isotopy classes of immersed circles on orientable surfaces. Motivated by the general idea of Artin and a recent work of Bellingeri and Paris \cite{BellingeriParis2020}, we obtain a complete description of homomorphisms between virtual twin groups and symmetric groups, which as an application gives us the precise structure of the automorphism group of the virtual twin group $VT_n$  on $n \ge 2$ strands. This is achieved by showing the existence of an irreducible right-angled Coxeter group $KT_n$ inside $VT_n$. As a by-product, it also follows that the twin group $T_n$ embeds inside the virtual twin group $VT_n$, which is an analogue of a similar result for braid groups.
\end{abstract}

\maketitle

\section{Introduction}
Doodles on a 2-sphere first appeared in the work \cite{FennTaylor} of Fenn and Taylor as finite collections of simple closed curves on a 2-sphere without triple or higher intersections. Allowing self intersections of curves, Khovanov \cite{Khovanov} extended the idea to finite collections of closed curves without triple or higher intersections on a closed oriented surface. Khovanov also introduced an analogue of the link group for doodles and constructed several infinite families of doodles whose fundamental groups have infinite centre. Recently, Bartholomew-Fenn-Kamada-Kamada \cite{BartholomewFennKamada2018, BartholomewFennKamada2018-2} extended the study of doodles to immersed circles without triple or higher intersection points on closed oriented surfaces, which can be thought of as a planar analogue of virtual knot theory with the sphere case corresponding to classical knot theory.  It is a natural problem to look for invariants for these topological objects. In \cite{BartholomewFennKamada2019}, coloring of diagrams using a special type of algebra has been used to construct an invariant for virtual doodles. Further, an Alexander type invariant for oriented doodles which vanishes on unlinked doodles with more than one component has been constructed in a recent work \cite{CisnerosFloresJuyumayaMarquez}.
\par

In tandem with classical knot theory, the study of doodles on surfaces is structured around a suitable group theory framework. The role of groups for doodles on a 2-sphere is played by a class of right-angled Coxeter groups called {\it twin groups} (also called {\it planar braid groups}), which first appeared in the work of Shabat and Voevodsky \cite{ShabatVoevodsky}. Twin groups have been brought to attention by Khovanov \cite{Khovanov} who gave a topological interpretation of these groups. For each $n \ge 2$, the twin group $T_n$ is the set of homotopy classes of configurations of $n$ arcs in the infinite strip $\mathbb{R} \times  [0,1]$ connecting fixed $n$ marked points on each of the parallel boundary lines such that each arc is monotonic and no three arcs have a point in common. The group structure on $T_n$ is given by the natural stacking operation. Taking the one point compactification of the plane, one can define the closure of a twin on a $2$-sphere analogous to the closure of a geometric braid in the 3-space. While Khovanov proved that every oriented doodle on a $2$-sphere is closure of a twin, an analogue of Markov Theorem for doodles on a 2-sphere is known due to Gotin \cite{Gotin}. A recent work \cite{NandaSingh2020} by Nanda and Singh established Alexander and Markov theorems for the virtual case. It is proved that a new class of groups called virtual twin groups, introduced in \cite{BarVesSin} and denoted by $VT_n$, plays the role of groups in the theory of virtual doodles. These correspondences can be summarised as
$$\bigcup_{n \ge 2} T_n/_{\textrm{Markov equivalence}}\quad  \longleftrightarrow \quad\textrm{Homotopy classes of doodles on 2-sphere}$$
and
$$\bigcup_{n \ge 2} VT_n/_{\textrm{Markov equivalence}} \quad\longleftrightarrow \quad\textrm{Stable equivalence classes of doodles on surfaces}.$$
\par

Analogues of pure braid groups and pure virtual braid groups can be defined for twin groups and virtual twin groups as well. The pure twin group $PT_n$ is defined as the kernel of the natural surjection from $T_n$ onto the symmetric group $S_n$. The structure of $PT_n$ is completely known for small number of strands. Bardakov, Singh and Vesnin \cite{BarVesSin} proved that $PT_n$ is free for $n = 3,4$ and not free for $n \geq 6$. Gonz\'alez, Le\'on-Medina and Roque \cite{GonGutiRoq} showed that $PT_5$ is a free group of rank $31$. A precise description of $PT_6$ has been obtained by Mostovoy and Roque-M\'arquez \cite{MostRoq} who proved that $PT_6 \cong F_{71} *_{20} \big(\mathbb{Z} \oplus \mathbb{Z} \big)$. Recently, a minimal presentation of $PT_n$ for all $n$ has been announced by Mostovoy \cite{Mostovoy}. Automorphisms, (twisted) conjugacy classes and centralisers of involutions in twin groups have been explored in recent works of the authors \cite{NaikNandaSingh1, NaikNandaSingh2}.  In a recent preprint \cite{Farley2021}, Farley has shown that $PT_n$ is always a diagram group, in the sense of Guba and Sapir. It is worth noting that (pure) twin groups are also used by physicists in the study of three-body interactions and topological exchange statistics in one dimension  \cite{HarshmanKnapp, HarshmanKnapp2021}. The pure virtual twin group $PVT_n$ is defined analogously as the kernel of the natural surjection from $VT_n$  onto  $S_n$. A precise presentation of $PVT_n$ has been obtained in a recent work \cite{NaikNandaSingh2020} of the authors, where it has been shown to be an irreducible right-angled Artin group. Further, a complete description of automorphism group of $PVT_n$ has been given. 
\par
The present paper contributes to our understanding of virtual twin groups and is motivated by the recent work \cite{BellingeriParis2020} of Bellingeri and Paris on virtual braid groups. We show that there exists an irreducible right-angled Coxeter group $KT_n$ inside the virtual twin group $VT_n$ and that $KT_n$ contains $T_n$. As a consequence, it follows that the twin group $T_n$ embeds inside the virtual twin group $VT_n$, which is an analogue of a similar but non-obvious result on embedding of braid groups inside virtual braid groups  \cite{MR1410467, Gaudreau2020, Kamada, MR1997331}. The group $KT_n$ is further used to obtain a complete description of homomorphisms between virtual twin groups and symmetric groups. It is worth pointing out that the study of homomorphisms from braid groups to symmetric groups goes back to Artin \cite{Artin}, which was later used by Dyer and Grossman \cite{DyerGrossman} to determine the automorphism groups of braid groups. The paper \cite{BellingeriParis2020} and our paper follows this general idea, although the techniques involved are quite different.
\medskip

We begin by recalling the definition and the topological interpretation of virtual twin groups in Section \ref{section-prelim}.  In Section \ref{section-presentation-pvtn}, we give a presentation of $KT_n$ showing that it is an irreducible right-angled Coxeter group. More precisely, we prove the following result (Theorem \ref{Ktn-right-angled-Coxeter}).

\begin{thm}
{\it For each $n \ge 2 $, the group $KT_n$  is generated by $\mathcal{S}= \big\{ \alpha_{i, j} ~|~  1 \leq i \neq j \leq n \big\}$, where $\alpha_{i, i+1}= s_i$ and  $\alpha_{i+1, i}= \rho_i s_i \rho_i$. Further, 
the defining relations are the following:
\begin{itemize}
\item[(1)] $\alpha_{i,j}^2 = 1$  for all $1\leq i \neq j \leq n,$
\item[(2)] $\alpha_{i,j} \alpha_{k,l} =  \alpha_{k,l} \alpha_{i,j}  \text{ for distinct integers } i, j, k, l$.
\end{itemize}}
\end{thm}

In Section \ref{homo from vtn to sn}, we give a complete description of homomorphisms from $VT_n$ to $S_m$ (Theorem \ref{Homomorphisms-VTn-Sm}). Section \ref{homo from sn to vtn} contains many technical results and occupies the main chunk of this paper. The main result of this section gives a complete description of homomorphisms from $S_n$ to $VT_m$ (Theorem \ref{Homomorphisms-Sn-VTn}). Finally, in Section \ref{Homomorphisms-VTn-VTm}, building upon the preceding sections, we give a complete description of homomorphisms from  $VT_n$ to $VT_m$. To be more specific, we establish the following result (Theorem \ref{homomorphisms-VTn-VTm}).

\begin{thm}
{\it Let $n, m$ be integers such that $n \geq m$, $n \geq 5$ and $m \geq 2$. Let $\phi : VT_n \to VT_m$ be a homomorphism. Then, upto conjugation of homomorphisms, one of the following  assertions holds:
\begin{enumerate}
\item The image of $\phi$ is abelian,
\item $n=m$ and $\phi \in \{\lambda  \pi,~ \lambda  \theta, ~\phi_m, ~ \zeta  \phi_m, ~\text{where}~ m \in \mathbb{Z} \}$,
\item $n=m=6$ and $\phi \in \{ \lambda  \nu  \theta, \lambda  \nu  \pi \}$.
\end{enumerate}}
\end{thm}

As a consequence, we obtain the structure of the automorphism group of $VT_n$ (Theorem  \ref{aut out vtn}) and prove that $\Aut(VT_n) \cong VT_n \rtimes \mathbb{Z}_2$ for $n \ge 5$. As an application, we deduce that $VT_n$ is not co-Hopfian for  $n \ge2$ (Corollary \ref{vtn not cohopfian}). We conclude the paper by tabulating status of some structural properties of braid groups, virtual braid groups, twin groups, virtual twin groups and their pure subgroups.
\medskip

\section{Preliminaries}\label{section-prelim}
Consider the group $VT_n$ with generators $\{ s_1, s_2, \ldots, s_{n-1}, \rho_1, \rho_2, \ldots, \rho_{n-1}\}$ and defining relations
\begin{eqnarray}
s_i^{2} &=&1 \quad \textrm{for} \quad i = 1, 2, \dots, n-1, \label{1}\\ 
s_is_j &=& s_js_i \quad\textrm{for} \quad |i - j| \geq 2,\label{2}\\
\rho_i^{2} &=& 1 \quad \textrm{for}\quad i = 1, 2, \dots, n-1, \label{3}\\
\rho_i\rho_j &=& \rho_j\rho_i \quad \textrm{for} \quad|i - j| \geq 2, \label{4}\\
\rho_i\rho_{i+1}\rho_i &=& \rho_{i+1}\rho_i\rho_{i+1} \quad \textrm{for}\quad i = 1, 2, \dots, n-2, \label{5}\\
\rho_is_j &=& s_j\rho_i \quad \textrm{for} \quad |i - j| \geq 2, \label{6}\\
\rho_i\rho_{i+1} s_i &=& s_{i+1} \rho_i \rho_{i+1} \quad \textrm{for} \quad i = 1, 2, \dots, n-2. \label{7}
\end{eqnarray}

Elements of the group $VT_n$ can be topologically interpreted as follows \cite{NandaSingh2020}.  Consider a subset $D$ of $\mathbb{R} \times [0,1]$ consisting of $n$ intervals called {\it strands} with $\partial (D) = Q_n  \times \{0,1\}$, where $Q_n$ is a  fixed set of $n$ points in $\mathbb{R}$. The set $D$ is called a virtual twin diagram on $n$ strands if it satisfies the following conditions.

\begin{enumerate}
\item[(1)]  Every strand is monotonic, more precisely, each strand maps homeomorphicaly onto the unit interval $[0,1]$ by the natural projection $\mathbb{R} \times [0,1] \to [0,1]$.
\item[(2)] The set $V(D)$ of all crossings of the diagram $D$ consists of transverse double points of $D$, where each crossing has the pre-assigned information of being a real or a virtual crossing as depicted in Figure \ref{Crossings}. A virtual crossing is depicted by a crossing encircled with a small circle.
\end{enumerate}

\begin{figure}[hbtp]
\centering
\includegraphics[scale=0.4]{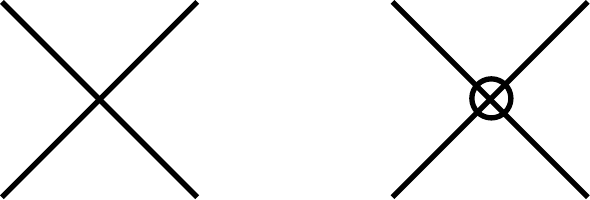}
\caption{Real and virtual crossings}
\label{Crossings}
\end{figure}

We say that the two virtual twin diagrams on $n$ strands are \textit{equivalent} if one can be obtained from the other by a finite sequence of isotopies of the plane and the moves as in Figure \ref{Reidemeister moves for virtual twin diagrams}. Such an equivalence class is called a \textit{virtual twin}. It turns out that $VT_n$ is isomorphic to the group of virtual twins on $n$ strands with the operation of concatenation \cite[Proposition 3.3]{NandaSingh2020}. The generators $s_i$ and $\rho_i$ of $VT_n$ can be represented by configurations shown in Figure \ref{generator-vtn}.

\begin{figure}[hbtp]
\centering
\includegraphics[scale=0.2]{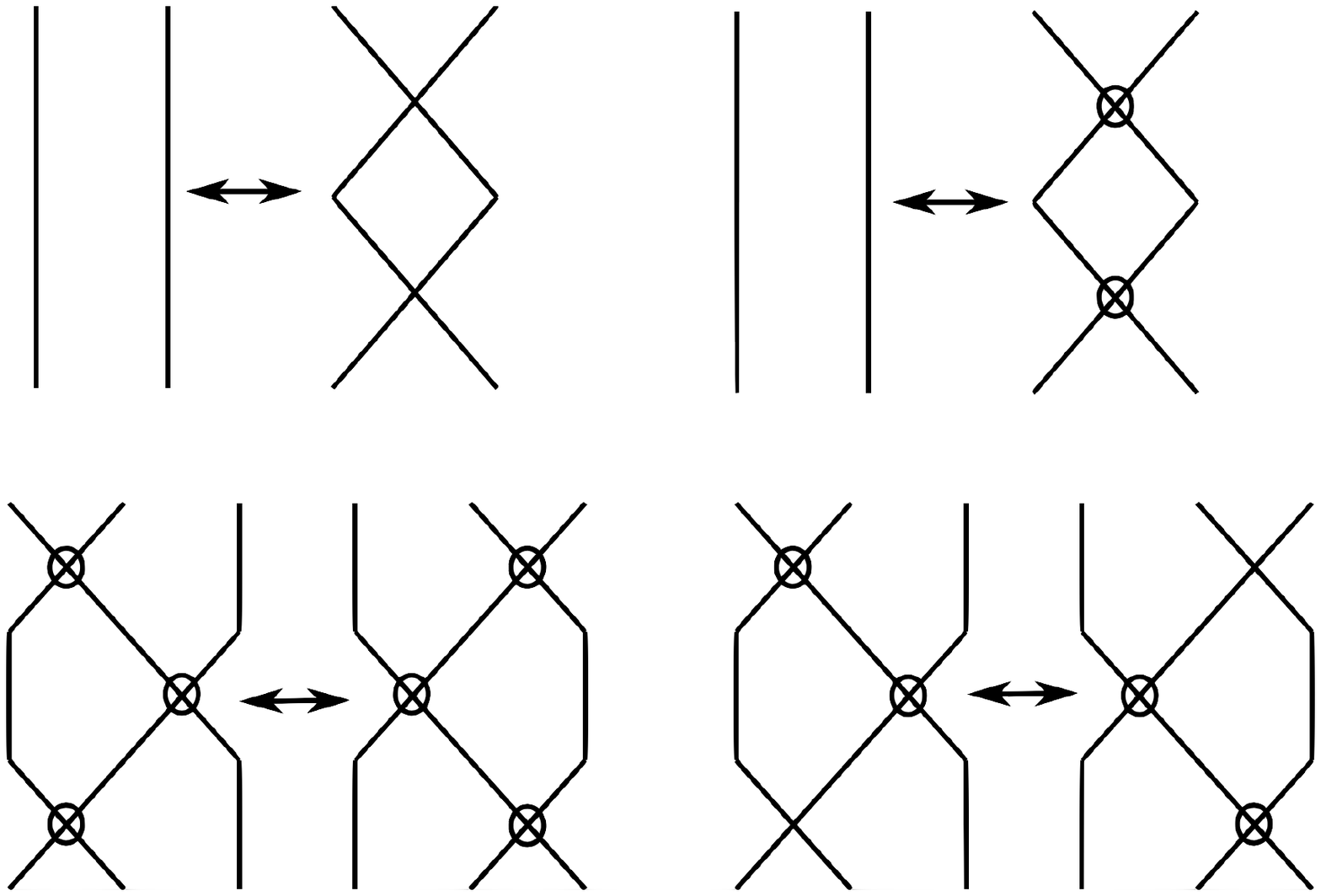}
\caption{Reidemeister moves for virtual twin diagrams}
\label{Reidemeister moves for virtual twin diagrams}
\end{figure}

\begin{figure}[hbtp]
\centering
\includegraphics[scale=0.25]{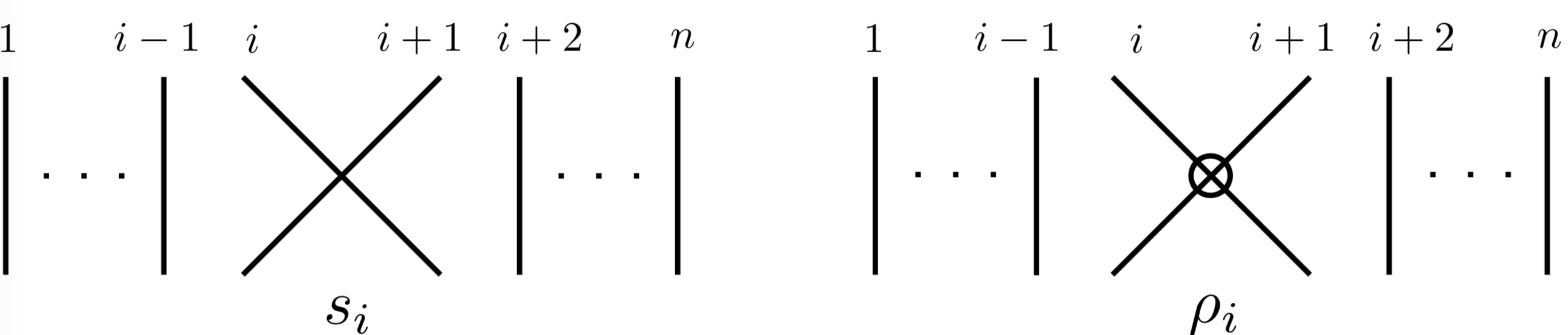}
\caption{Generator $s_i$ and $\rho_i$}
\label{generator-vtn}
\end{figure}

Let $\tau_i$ denote the transposition $(i, i+1)$. The symmetric group $S_n$ on $n$ symbols is generated by $\tau_1,\tau_2, \ldots, \tau_{n-1}$. There is a natural surjective homomorphism $\pi:VT_n  \to S_n$ given by $$\pi(s_i) = \pi(\rho_i) = \tau_i$$ for all $i$. The kernel $PVT_n$ of this surjection is called the \textit{pure virtual twin group} on $n$ strands.
\medskip

There is another surjective group homomorphism $\theta: VT_n  \to S_n$ given by
$$\theta(s_i) = 1 \quad \text{and} \quad \theta(\rho_i) = \tau_i$$
for all $i$. We denote the kernel of this surjection by $KT_n$. This group plays a crucial role in the rest of this paper. The map $\lambda: S_n \to VT_n$ given by $\lambda(\tau_i)= \rho_i$ is a splitting of the short exact sequence
$$1 \to KT_n \to VT_n \to S_n \to 1,$$
and hence $VT_n = KT_n \rtimes S_n$.
\par
The {\it twin group} $T_n$ has generators $\{ s_1, s_2, \ldots, s_{n-1}\}$ and defining relations
\begin{eqnarray*}
s_i^{2} &=&1 \quad \textrm{for } i = 1, 2, \dots, n-1\\
s_is_j &=& s_js_i \quad \textrm{for } |i - j| \geq 2.
\end{eqnarray*}
It is not clear immediately whether $T_n$ is a subgroup of $VT_n$. We shown later in Corollary \ref{tn-in-ptn} that this is indeed the case.
\par
Throughout,  $\widehat{x}$ denote the inner automorphism of a group $G$ induced by an element $x \in G$. To be precise,  $\widehat{x}(y)=x y x^{-1}$ for all $y\in G$. As usual, the commutator $xy x^{-1}y^{-1}$ is denoted by $[x,y]$. The centraliser of a subgroup $H$ of $G$ is denoted by $\C_G(H)$.
\medskip

\section{Presentation of $KT_n$}\label{section-presentation-pvtn}
In this section, we give a presentation of $KT_n$. We use the standard presentation of $VT_n$ from Section \ref{section-prelim} and the Reidemeister-Schreier method \cite[Theorem 2.6]{Magnus1966}. We take the set 
$$\M_n = \big\{ m_{1, i_1}m_{2, i_2}\dots m_{n-1, i_{n-1}} ~|~ m_{k, i_k}=\rho_k \rho_{k-1}\dots \rho_{i_{k}+1}~ \text{ for each } ~1 \leq k \leq n-1 \text{ and } 0 \leq i_k < k  \big\}$$
as the Schreier system of coset representatives of $KT_n$ in $VT_n$. We set $m_{kk}=1$ for $1 \leq k \leq n-1$. For an element $w \in VT_n$, let $\overline{w}$ denote the unique coset representative of the coset of $w$ in the Schreier set $\M_n$. By Reidemeister-Schreier method, the group $KT_n$ is generated by the set
$$\big\{ \gamma( \mu, a) = (\mu a) (\overline{\mu a})^{-1} ~|~ \mu \in \M_n \text{ and } a \in \{s_1, \dots, s_{n-1}, \rho_1, \dots, \rho_{n-1}\} \big\}.$$
with defining relations
$$\big\{\tau(\mu r \mu^{-1}) \mid \mu \in \M_n~\text{and}~ r ~\text{is a defining relation in} ~VT_n \big\},$$
where $\tau$ is the rewriting process. More precisely, for an element $g=g_1g_2\dots g_k \in VT_n$, we have 
$$\tau(g) = \gamma(1, g_1)\gamma(\overline{g_1}, g_2)\dots \gamma(\overline{g_1g_2\dots g_{k-1}}, g_k).$$

For each $1 \le i \le n-1$, we set
$$\alpha_{i, i+1}= s_i \quad \text{and} \quad \alpha_{i+1, i}= \rho_i s_i \rho_i.$$
For each $1 \leq i < j \leq n$ and $j \ne i+1$, we set
$$\alpha_{i,j} = (\rho_{j-1} \rho_{j-2} \dots \rho_{i+1}) \alpha_{i, i+1} (\rho_{i+1} \dots \rho_{j-2}  \rho_{j-1}) = (\rho_{j-1} \rho_{j-2} \dots \rho_{i+1}) s_i (\rho_{i+1} \dots \rho_{j-2}  \rho_{j-1}),$$ 
and
$$\alpha_{j,i} = (\rho_{j-1} \rho_{j-2} \dots \rho_{i+1}) \alpha_{i+1, i} (\rho_{i+1} \dots \rho_{j-2}  \rho_{j-1})= (\rho_{j-1} \rho_{j-2} \dots \rho_{i+1})\rho_i s_i \rho_i(\rho_{i+1} \dots \rho_{j-2}  \rho_{j-1}).$$
For each $n \ge 2$, let us define
$$\mathcal{S}= \big\{ \alpha_{i, j} ~|~  1 \leq i \neq j \leq n \big\}.$$

For each pair $(i, j)$ with $1 \leq i < j \leq n$, the generators $\alpha_{i,j}$ and $\alpha_{j,i}$ can be topologically represented as in Figure \ref{Generator-KTn}.
\begin{figure}[hbtp]
\centering
\includegraphics[scale=0.9]{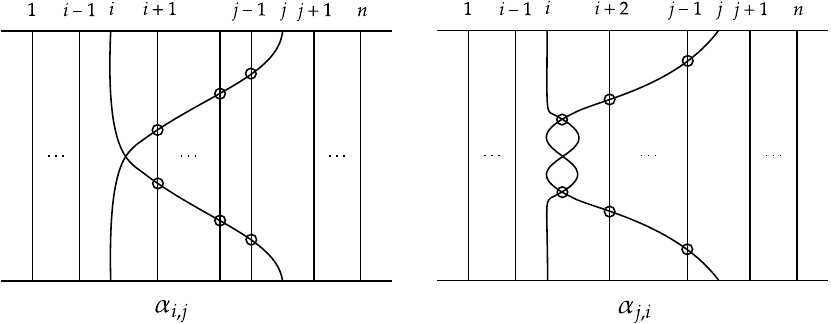}
\caption{Generators $\alpha_{i,j}$ and $\alpha_{j,i}$ of $KT_n$}
\label{Generator-KTn}
\end{figure}

\begin{theorem}\label{KTn-generators}
For each $n \ge 2 $,  the  group $KT_n$ is generated by $\mathcal{S}= \big\{ \alpha_{i, j} ~|~  1 \leq i \neq j \leq n \big\}.$
\end{theorem}

\begin{proof}
The case $n=2$ is immediate, and hence we assume $n \ge 3$. Note that $KT_n$ is generated by elements $\gamma( \mu, a) = (\mu a) (\overline{\mu a})^{-1}$, where $\mu \in \M_n$ and $a \in \{s_1, \dots, s_{n-1}, \rho_1, \dots, \rho_{n-1}\}$. Let $w=w_1w_2\ldots w_k$, where $w_i \in \{s_1, \dots, s_{n-1}, \rho_1, \dots, \rho_{n-1}\}$. Then, we have $\overline{w} = w_1^*w_2^*\ldots w_k^*$, where $w_i^* = w_i$ if $w_i \in \{\rho_1, \dots, \rho_{n-1}\}$, and $w_i^* = 1$ if $w_i \in \{s_1, \dots, s_{n-1}\}$. Thus, for each  $\mu \in \M_n$ and $1 \le i \le n-1$, we have 
$$\gamma( \mu, \rho_i)  = (\mu \rho_i) (\overline{\mu \rho_i})^{-1} = (\mu \rho_i) (\mu \rho_i)^{-1} = 1$$
and
$$\gamma( \mu, s_i) = (\mu s_i) (\overline{\mu s_i})^{-1} = \mu s_i\mu ^{-1} = \mu \alpha_{i, i+1} \mu^{-1}.$$

We claim that each $\gamma( \mu, s_i)$ lie in  $\mathcal{S}$ and that the conjugation action of $\langle \rho_1, \ldots, \rho_{n-1} \rangle \cong S_n$ on $\mathcal{S}$ is transitive.
\par 

First consider $\alpha_{i,i+1}$ and $\alpha_{i+1,i}$ for a fixed $1\leq i \leq n-1$.

\begin{multicols}{2}
\begin{itemize}[leftmargin=*]
\item $1 \leq k \leq i-2$ or $i+2 \leq k \leq n-1$: 
\begin{align*}
 \rho_k \alpha_{i,i+1} \rho_k &= \alpha_{i,i+1},&\\
 \rho_k \alpha_{i+1,i} \rho_k &= \alpha_{i+1,i}.&
\end{align*}
 \item $ k= i-1:$
\begin{align*}
 \rho_k \alpha_{i,i+1} \rho_k &= \rho_{i-1} s_i \rho_{i-1} = \rho_i s_{i-1} \rho_i = \alpha_{i-1,i+1},\\
 \rho_k \alpha_{i+1,i} \rho_k
 &= \rho_i\rho_{i-1} (\rho_i \rho_{i-1} s_i \rho_{i-1} \rho_i)\rho_{i-1} \rho_i\\
  &=\alpha_{i+1,i-1}.
 \end{align*}
\end{itemize}

\columnbreak

\begin{itemize}[leftmargin=*]
\item $k=i:$
\begin{align*}
 \rho_k \alpha_{i,i+1} \rho_k &= \rho_i s_i \rho_i = \alpha_{i+1,i} ,&\\
 \rho_k \alpha_{i+1,i} \rho_k &=  \rho_i \rho_i s_i \rho_i \rho_i = s_i = \alpha_{i,i+1}.&
\end{align*}

\item $k=i+1:$
\begin{align*}
 \rho_k \alpha_{i,i+1} \rho_k &= \rho_{i+1} s_i \rho_{i+1} = \alpha_{i,i+2} ,&\\
 \rho_k \alpha_{i+1,i} \rho_k &=  \rho_{i+1} \rho_i s_i \rho_i \rho_{i+1} =  \alpha_{i+2,i}.&
\end{align*}
\end{itemize}
\end{multicols}

Next, we consider $\alpha_{i,j}$ and $\alpha_{j,i}$ for some fixed $1 \leq i < j \leq n$ with $j \ne i+1$. 
\begin{itemize}[leftmargin=*]
\item $1 \leq k \leq i-2$ or $j+1 \leq k \leq n-1:$

\begin{align*}
 \rho_k \alpha_{i,j} \rho_k &= \alpha_{i,j},& \rho_k \alpha_{j,i} \rho_k &= \alpha_{j,i}.&
\end{align*}
\item $k=i-1:$
\begin{align*}
\rho_k \alpha_{i,j} \rho_k &= \rho_{i-1} (\rho_{j-1} \rho_{j-2} \dots \rho_{i+1}) s_i (\rho_{i+1} \dots \rho_{j-2}  \rho_{j-1}) \rho_{i-1}\\
&= (\rho_{j-1} \rho_{j-2} \dots \rho_{i+1})  \rho_i s_{i-1} \rho_i (\rho_{i+1} \dots \rho_{j-2}  \rho_{j-1})\\
&= \alpha_{i-1,j},
\end{align*}
\begin{align*}
\rho_k \alpha_{j,i} \rho_k &= \rho_{i-1} (\rho_{j-1} \rho_{j-2} \dots \rho_{i+1}) \rho_i s_i \rho_i (\rho_{i+1} \dots \rho_{j-2}  \rho_{j-1}) \rho_{i-1}\\
&= (\rho_{j-1} \rho_{j-2} \dots \rho_{i+1}) \rho_i\rho_{i-1} (\rho_i  \rho_{i-1}  s_i \rho_{i-1} \rho_i)\rho_{i-1} \rho_i (\rho_{i+1} \dots \rho_{j-2}  \rho_{j-1})\\
&= (\rho_{j-1} \rho_{j-2} \dots \rho_{i+1}) \rho_i\rho_{i-1} s_{i-1}\rho_{i-1} \rho_i (\rho_{i+1} \dots \rho_{j-2}  \rho_{j-1})\\
&= \alpha_{j,i-1}.
\end{align*}
\item $k=i:$
\begin{align*}
\rho_k \alpha_{i,j} \rho_k &= \rho_i (\rho_{j-1} \rho_{j-2} \dots \rho_{i+1}) s_i (\rho_{i+1} \dots \rho_{j-2}  \rho_{j-1}) \rho_i\\
&=  (\rho_{j-1} \rho_{j-2} \dots \rho_{i+2})\rho_i \rho_{i+1} s_i \rho_{i+1} \rho_i(\rho_{i+2} \dots \rho_{j-2}  \rho_{j-1})\\
&=  (\rho_{j-1} \rho_{j-2} \dots \rho_{i+2})s_{i+1}(\rho_{i+2} \dots \rho_{j-2}  \rho_{j-1})\\
&= \alpha_{i+1,j},\\
\rho_k \alpha_{j,i} \rho_k &= \rho_i (\rho_{j-1} \rho_{j-2} \dots \rho_{i+1}) \rho_i s_i \rho_i (\rho_{i+1} \dots \rho_{j-2}  \rho_{j-1}) \rho_i\\
&=  (\rho_{j-1} \rho_{j-2} \dots \rho_{i+2})\rho_i \rho_{i+1} \rho_i s_i \rho_i \rho_{i+1} \rho_i(\rho_{i+2} \dots \rho_{j-2}  \rho_{j-1})\\
&=  (\rho_{j-1} \rho_{j-2} \dots \rho_{i+2})\rho_{i+1} s_{i+1} \rho_{i+1}(\rho_{i+2} \dots \rho_{j-2}  \rho_{j-1})\\
&= \alpha_{j,i+1}.
\end{align*}
\item $k= j-1:$
\begin{align*}
\rho_k \alpha_{i,j} \rho_k &= \rho_{j-1} \rho_{j-1} \dots \rho_{k+1} \rho_{k} \dots \rho_{i+1} s_i \rho_{i+1} \dots \rho_k \rho_{k+1} \dots \rho_{j-1} \rho_{j-1} \\
&= \rho_{j-2} \dots \rho_{k+1} \rho_{k} \dots \rho_{i+1} s_i \rho_{i+1} \dots \rho_k \rho_{k+1} \dots \rho_{j-2} \\
&= \alpha_{i,j-1}
\end{align*}
\begin{align*}
\rho_k \alpha_{j,i} \rho_k &= \rho_{j-1} \rho_{j-1} \dots \rho_{k+1} \rho_{k} \dots \rho_{i+1} \rho_i s_i \rho_i \rho_{i+1} \dots \rho_k \rho_{k+1} \dots \rho_{j-1} \rho_{j-1} \\
&= \rho_{j-2} \dots \rho_{k+1} \rho_{k} \dots \rho_{i+1} \rho_i s_i \rho_i  \rho_{i+1} \dots \rho_k \rho_{k+1} \dots \rho_{j-2} \\
&= \alpha_{j-1,i}.
\end{align*}
\item $k= j:$
\begin{align*}
\rho_k \alpha_{i,j} \rho_k = \alpha_{i,j+1},\\
\rho_k \alpha_{j,i} \rho_k = \alpha_{j+1,i}.
\end{align*}
\item $ i+1 \leq k \leq j-2:$
\begin{align*}
\rho_k \alpha_{i,j} \rho_k &= \rho_k \rho_{j-1} \dots \rho_{k+1} \rho_{k} \dots \rho_{i+1} s_i \rho_{i+1} \dots \rho_k \rho_{k+1} \dots \rho_{j-1} \rho_k \\
&= (\rho_{j-1} \dots \rho_{k+2}) \rho_k \rho_{k+1} \rho_{k} (\rho_{k-1} \dots \rho_{i+1}) s_i (\rho_{i+1} \dots \rho_{k-1})\rho_k \rho_{k+1} \rho_k  (\rho_{k+2}\dots \rho_{j-1})\\
&= (\rho_{j-1} \dots \rho_{k+2}) \rho_{k+1} \rho_k \rho_{k+1}  (\rho_{k-1} \dots \rho_{i+1}) s_i (\rho_{i+1} \dots \rho_{k-1})\rho_{k+1} \rho_k \rho_{k+1}   (\rho_{k+2}\dots \rho_{j-1})\\
&= (\rho_{j-1} \dots \rho_{k+2} \rho_{k+1} \rho_k \rho_{k-1} \dots \rho_{i+1}) \rho_{k+1} s_i \rho_{k+1} (\rho_{i+1} \dots \rho_{k-1} \rho_k \rho_{k+1}  \rho_{k+2}\dots \rho_{j-1})\\
&= (\rho_{j-1} \dots \rho_{k+2} \rho_{k+1} \rho_k \rho_{k-1} \dots \rho_{i+1})  s_i  (\rho_{i+1} \dots \rho_{k-1} \rho_k \rho_{k+1}  \rho_{k+2}\dots \rho_{j-1})\\
&= \alpha_{i,j},\\
\rho_k \alpha_{j,i} \rho_k &= \rho_k \rho_{j-1} \dots \rho_{k+1} \rho_{k} \dots \rho_{i+1} \rho_i s_i \rho_i \rho_{i+1} \dots \rho_k \rho_{k+1} \dots \rho_{j-1} \rho_k \\
&= (\rho_{j-1} \dots \rho_{k+2}) \rho_k \rho_{k+1} \rho_{k} (\rho_{k-1} \dots \rho_{i+1}) \rho_i s_i \rho_i  (\rho_{i+1} \dots \rho_{k-1})\rho_k \rho_{k+1} \rho_k  (\rho_{k+2}\dots \rho_{j-1})\\
&= (\rho_{j-1} \dots \rho_{k+2} \rho_{k+1} \rho_k \rho_{k-1} \dots \rho_{i+1}) \rho_{k+1} \rho_i s_i \rho_i  \rho_{k+1} (\rho_{i+1} \dots \rho_{k-1} \rho_k \rho_{k+1}  \rho_{k+2}\dots \rho_{j-1})\\
&= (\rho_{j-1} \dots \rho_{k+2} \rho_{k+1} \rho_k \rho_{k-1} \dots \rho_{i+1} )\rho_i s_i \rho_i(\rho_{i+1} \dots \rho_{k-1} \rho_k \rho_{k+1}  \rho_{k+2}\dots \rho_{j-1})\\
&= \alpha_{j,i}.
\end{align*}
\end{itemize}

Hence, each generator $\gamma( \mu, s_i)$ lie in $\mathcal{S} $. Conversely, if $1 \leq i< j \leq n$, then we see that conjugation by $(\rho_{i-1}\rho_{i-2}\ldots \rho_2\rho_1)(\rho_{j-1}\rho_{j-2}\ldots \rho_3\rho_2)$ maps 
$\alpha_{1,2}$ (respectively $\alpha_{2,1}$) to $\alpha_{i,j}$, (respectively $\alpha_{j,i}$) whereas conjugation by $\rho_1$ maps $\alpha_{1,2}$ to $\alpha_{2,1}$. That is, the conjugation action of  $S_n$ on the set $\mathcal{S}$ is transitive. Hence, we have proved that $\mathcal{S}$ generates $KT_n$.
\end{proof}

\begin{remark}\label{rephrasing-action}
We can summarise the (left) action of $S_n$ on the set  $\mathcal{S}$ as
\begin{equation*}
\tau_k \cdot \alpha_{i,j}:= \rho_k \alpha_{i,j} \rho_k = \alpha_{\rho_k(i),\rho_k(j)}
\end{equation*}
for every $1 \leq i \neq j \leq n$ and $1 \le k \le n-1$.
\medskip
\end{remark}

\begin{theorem}\label{Ktn-right-angled-Coxeter}
For each $n \ge 2 $, the group $KT_n$  is generated by $\mathcal{S}= \big\{ \alpha_{i, j} ~|~  1 \leq i \neq j \leq n \big\}$ with the following defining relations:
\begin{itemize}
\item[(1)] $\alpha_{i,j}^2 = 1$  for all $1\leq i \neq j \leq n,$ and
\item[(2)] $\alpha_{i,j} \alpha_{k,l} =  \alpha_{k,l} \alpha_{i,j}  \text{ for distinct integers } i, j, k, l$.
\end{itemize}
\end{theorem}

\begin{proof}
Theorem \ref{KTn-generators} already shows that $\mathcal{S}$ generates $KT_n$. The defining relations are given by 
$$\tau(\mu r \mu^{-1}),$$
where $\tau$ is the rewriting process, $\mu \in \M_n$ and $r$ is a defining relation in $VT_n$.
\medskip

Let us take $ \mu = \rho_{i_1} \rho_{i_2} \dots \rho_{i_k} \in \M_n$ and $g=g_1g_2\dots g_t$ a relation of $VT_n$. Note that, since $\gamma(\mu, \rho_i)=1$ for all $i$,  we have 
$$\tau(\mu g \mu^{-1}) = \gamma(\overline{\mu}, g_1)\gamma(\overline{\mu g_1}, g_2)\dots \gamma(\overline{\mu g_1g_2\dots g_{t-1}}, g_t).$$
Further, no non-trivial relations for $KT_n$ can be obtained from the relations \eqref{3}--\eqref{5} of $VT_n$. Next, we consider the remaining relations one by one.
\begin{itemize}
\item First we consider the relations $s_i^2 = 1$ for $1 \le i \le n-1$. In this case, we have
\begin{align*}
\tau(\mu s_i^2 \mu^{-1}) &=  \gamma(\overline{\mu}, s_i) \gamma(\overline{\mu s_i}, s_i)\\
&= \gamma(\mu, s_i) \gamma(\mu, s_i) = (\mu \alpha_{i,i+1}  \mu^{-1})^2.
\end{align*}
Since the conjugation action of $S_n$ on $\mathcal{S}$ is transitive, it follows that all the generators are involutions.
\item[]

\item Next we consider the relations $ (s_i \rho_j)^2 = 1$ for $|i-j| > 1$. We have
\begin{align*}
\tau(\mu s_i \rho_j s_i \rho_j \mu^{-1}) &= \gamma(\overline{\mu}, s_i) \gamma(\overline{\mu s_i \rho_j}, s_i)\\
&= \gamma(\mu, s_i) \gamma(\mu \rho_j, s_i) = (\mu s_i  \mu^{-1}) (\mu \rho_j s_i \rho_j  \mu^{-1}) = (\mu s_i  \mu^{-1})^2,
\end{align*}
which again shows that the generators are of order two.
\item[]
\item Now we consider the relations $\rho_i s_{i+1} \rho_i \rho_{i+1} s_i \rho_{i+1}=1$, where $1 \le i\le n-2$. Computing
\begin{align*}
\tau(\mu \rho_i s_{i+1} \rho_i \rho_{i+1} s_i \rho_{i+1} \mu^{-1})&= \gamma(\overline{\mu \rho_i}, s_{i+1}) \gamma(\overline{ \mu \rho_i s_{i+1} \rho_i \rho_{i+1}}, s_i)\\
&= \gamma(\mu \rho_i, s_{i+1}) \gamma(\mu  \rho_{i+1}, s_i) = (\mu \rho_i s_{i+1} \rho_i \mu^{-1}) (\mu \rho_{i+1} s_i \rho_{i+1} \mu^{-1})\\
&= (\mu \rho_i \alpha_{i+1, i+2}  \rho_i\mu^{-1}) (\mu \rho_{i+1}  \alpha_{i, i+1} \rho_{i+1} \mu^{-1}) = (\mu \alpha_{i, i+2} \mu^{-1}) (\mu  \alpha_{i, i+2}  \mu^{-1}),
\end{align*}
we see that the generators are of order two.
\item[]
\item Finally we consider the relations $ (s_i s_j)^2 = 1$ for $|i-j| > 1$. If $\mu = 1$, then we have
\begin{align*}
\tau(s_i s_j s_i s_j) &= \gamma(1, s_i) \gamma(\overline{s_i}, s_j)  \gamma(\overline{s_i s_j}, s_i)  \gamma(\overline{s_i s_j s_i}, s_j)\\
&= \gamma(1, s_i) \gamma(1, s_j)  \gamma(1, s_i)  \gamma(1, s_j) = (\alpha_{i, i+1} \alpha_{j, j+1})^2.
\end{align*}
For $\mu \neq 1$, we have
\begin{align}
\nonumber \tau(\mu s_i s_j s_i s_j \mu^{-1}) &= \gamma(\overline{\mu}, s_i) \gamma(\overline{\mu s_i}, s_j)  \gamma(\overline{\mu s_i s_j}, s_i)  \gamma(\overline{ \mu s_i s_j s_i}, s_j)\\
\nonumber &= \gamma(\mu, s_i) \gamma(\mu, s_j)  \gamma(\mu, s_i)  \gamma(\mu, s_j) = (\mu s_i \mu^{-1})(\mu s_j \mu^{-1}) (\mu s_i \mu^{-1}) (\mu s_j \mu^{-1})\\
&= ((\mu \alpha_{i, i+1} \mu^{-1})(\mu \alpha_{j, j+1} \mu^{-1}))^2.\label{commuting-relations}
\end{align}
\end{itemize}
Note that for $n=2, 3$, these types of relations do not occur. For $n \ge 4$, we set 
$$\mathcal{D} = \big\{(\alpha_{i,j}, \alpha_{k,l})~|~i, j, k, l~\textrm{are distinct integers between } 1 \textrm{ and } n\big\}.$$
Remark \ref{rephrasing-action} and the theory of symmetric groups gives an induced transitive action of $S_n$ on $\mathcal{D}$ given by 
 $$ \rho \cdot (\alpha_{i,j}, \alpha_{k,l}) =  (\alpha_{\rho(i),\rho(j)}, \alpha_{\rho(k),\rho(l)})$$
 for all $\rho \in S_n$. Thus, the  defining relations of $KT_n$ obtained from \eqref{commuting-relations} are precisely of the form 
$$\alpha_{i,j} \alpha_{k,l} =  \alpha_{k,l} \alpha_{i,j},$$
where $i, j, k, l$ are distinct integers between 1 and $n$. This completes the proof of the theorem.
\end{proof}

\begin{corollary}\label{irred-raag}
For each $n \ge 2$, the  group $KT_n$ is an irreducible right-angled Coxeter group of rank $n(n-1)$ and with trivial center.
\end{corollary}

\begin{proof}
That $KT_n$ is a right-angled Coxeter group of rank $n(n-1)$ follows from Theorem \ref{Ktn-right-angled-Coxeter}. Irreducibility follows from the corresponding Coxeter graph of $KT_n$.
By Bourbaki \cite[p.137]{Bourbaki}, the center of an infinite irreducible Coxeter group is trivial.
\end{proof}

Recall that, if $(W, S)$ is a Coxeter system and $X$ a subset of $S$, then the subgroup of $W$ generated by $X$ is called a standard parabolic subgroup of $W$, and is denoted by $W[X]$.  It is well-known that the Artin braid group $B_n$ embeds inside the virtual braid group $VB_n$ \cite{MR1410467, Gaudreau2020, Kamada, MR1997331}. As another consequence of Theorem \ref{Ktn-right-angled-Coxeter}, we obtain a similar result for twin and virtual twin groups.

\begin{corollary}\label{tn-in-ptn}
$T_n$ is a subgroup of $VT_n$ for each $n \ge 2$.
\end{corollary}

\begin{proof}
The standard parabolic subgroup of $KT_n$ generated by $$\{ \alpha_{i,i+1} = s_i \mid 1\leq i \leq n-1\}$$ is precisely the twin group $T_n$. Hence, $T_n$ sits inside $KT_n$, and consequently inside $VT_n$.
\end{proof}

Recall that the {\it pure twin group} $PT_n$ is the kernel of the natural surjection from $T_n$ onto $S_n$ given by $s_i \mapsto \tau_i$. It follows from Corollary \ref{tn-in-ptn} that $PT_n$ is a subgroup of $PVT_n$, where it has been proved recently that $PVT_n$ is a right-angled Artin group \cite[Corollary 3.4]{NaikNandaSingh2020}. As noted in the introduction, $PT_n$ is free for $n = 3,4, 5$, and $PT_6$ is isomorphic to the free product of $F_{71}$ and  20 copies of $\mathbb{Z} \oplus \mathbb{Z}$, which are all right-angled Artin groups. Though a  presentation of $PT_n$ has been given in  \cite{Mostovoy}, it is not clear whether $PT_n$ is a right-angled Artin group for $n \ge 7$, but we believe that it is the case.
\begin{conjecture}
$PT_n$ is a right-angled Artin group for each $n \ge 3$.
\end{conjecture}
\medskip

\section{Homomorphisms from $VT_n$ to $S_m$}\label{homo from vtn to sn}
A group homomorphism $\psi: G \to H$ is said to be \textit{abelian} if $\psi(G)$ is an abelian subgroup of $H$. Two homomorphisms $\psi_1, \psi_2 : G \to H$ are said to be \textit{conjugate} if there exists $x \in H$ such that $\psi_2 = \widehat{x} ~\psi_1$, where $\widehat{x}$ is the inner automorphism induced by $x$, as defined at the end of Section \ref{section-prelim}. It is to be noted that $\Out(S_n)$ is trivial for all $n \neq 6$ and $\Out(S_6) \cong \mathbb{Z}_2$. The latter group is generated by the class of a non-inner automorphism $\nu : S_6 \to S_6$ of order two.
\medskip

The following result is well-known from the works of Artin \cite{Artin} and Lin \cite{Lin-1, Lin-2} and is crucial for the proof of Theorem \ref{Homomorphisms-VTn-Sm}.

\begin{proposition}\label{Homorphisms-Sn-Sm}
Let $n, m$ be integers such that $n \geq m$, $n \geq 5$ and $m \geq 2$. Let $\phi : S_n \to S_m$ be a homomorphism. Then, upto conjugation of homomorphisms, one of the following assertions holds:
\begin{enumerate}
\item $\phi$ is abelian,
\item $n=m$ and $\phi = \id$, 
\item $n=m=6$ and $\phi = \nu$.
\end{enumerate}
\end{proposition}

Let $\theta: VT_n  \to S_n$ and  $\lambda: S_n \rightarrow VT_n$ be as defined in Section \ref{section-prelim}. We prove the following result. 

\begin{theorem}\label{Homomorphisms-VTn-Sm}
Let $n, m$ be integers such that $n \geq m$, $n \geq 5$ and $m \geq 2$. Let $\psi : VT_n \to S_m$ be a homomorphism. Then, upto conjugation of homomorphisms, one of the following  assertions holds:
\begin{enumerate}
\item $\psi$ is abelian,
\item $n=m$ and $\psi = \pi$ or $\theta$, 
\item $n=m=6$ and $\psi = \nu  \pi$ or $\nu  \theta$.
\end{enumerate}
\end{theorem}

\begin{proof}
Consider the composition $S_n \stackrel{\lambda}{\longrightarrow}  VB_n  \stackrel{\psi}{\longrightarrow}  S_m.$ By Proposition \ref{Homorphisms-Sn-Sm}, one of the following holds for $\psi  \lambda$: 
\begin{enumerate}
\item $\psi  \lambda$ is abelian,
\item $n=m$ and $\psi  \lambda = \id$,
\item $n=m=6$ and $\psi  \lambda = \nu$.
\end{enumerate}
\medskip

Case (1): Let $\psi  \lambda$ be abelian. We claim that there exists $w \in S_m$ such that for $\psi  \lambda(\tau_i)=w$ for all $1\le i \le n-1$. Suppose on the contrary that
there exist $i$ and $w_1 \neq w_2$ in $S_m$ such that $\psi  \lambda(\tau_i)=w_1$ and $\psi  \lambda(\tau_{i+1})=w_2$. The braid relation $\tau_i\tau_{i+1}\tau_i = \tau_{i+1}\tau_i\tau_{i+1}$ gives $w_1w_2w_1= w_2w_1w_2$. Since $\psi  \lambda$ is abelian, we must have $w_1 = w_2$, a contradiction. This proves the claim. Next, we find $\psi(s_i)$. The relation $\rho_i s_{i+1} \rho_i = \rho_{i+1} s_i \rho_{i+1}$ gives $\psi(s_i) = \psi(s_{i+1})=z$ (say) for all $i$. Finally, the relation $s_1 \rho_3 = \rho_3 s_1$ gives $zw=wz$, and hence $\psi$ is abelian. 
\medskip

Case (2): Suppose that $n=m$ and $\psi  \lambda=\id$. In this case, we have $\psi(\rho_i) = \tau_i$ for all $1 \le i \le n-1$. Next, we need to find $\psi(s_i)$. Recall the relation $s_1 \rho_i= \rho_i s_1$ for $3 \le i \le n-1$. It follows that $\psi(s_1) \in \langle \tau_1 \rangle$, the centraliser of the subgroup $\langle \tau_3, \tau_4, \dots, \tau_{n-1} \rangle$ in $S_n$. Thus, we have either $\psi(s_1)=1$ or $\psi(s_1)=\tau_1$. If $\psi(s_1)=1$, then the relation $\rho_1 s_2 \rho_1 = \rho_2 s_1 \rho_2$ gives $\psi(s_2)=1$, and consequently $\psi(s_i)=1$ for all $i$. Thus, we obtain $\psi= \theta.$ And, if $\psi(s_1)=\tau_1$, then the  relation $\rho_1 s_2 \rho_1 = \rho_2 s_1 \rho_2$ gives $\tau_1 \psi(s_2) \tau_1= \tau_2 \tau_1 \tau_2 =\tau_1 \tau_2 \tau_1$. Thus, we get $\psi(s_2) = \tau_2$, and consequently $\psi(s_i)=\tau_i$ for all $i$. Thus, in this case $\psi= \pi.$
\medskip

Case (3): Suppose that $n=m=6$ and that $\psi  \lambda=\nu$. Then we have $\nu^{-1}  \psi  \lambda=\id.$ By Case (2), we have $\nu^{-1}  \psi= \pi$ or $\nu^{-1}  \psi = \theta$, and hence $\psi=\nu  \pi$ or $\psi=\nu  \theta.$
\end{proof}
\medskip

\section{Homomorphisms from $S_n$ to $VT_m$}\label{homo from sn to vtn}
This section occupies most of the remaining part of the paper. For notational convenience, for the rest of the paper, we set $K_n:=KT_n$ for each $n \ge 2$. Recall that $K_n$ is a right-angled Coxeter group with a Coxeter generating set $\mathcal{S} = \{ \alpha_{i,j} \mid 1\leq i \neq j \leq n \}$ and defining relations
\begin{enumerate}
\item $\alpha_{i,j}^2 = 1$ for all $1\leq i \neq j \leq n,$ and
\item $\alpha_{i,j} \alpha_{k,l} = \alpha_{k,l} \alpha_{i,j}$ for distinct integers $1\leq i, j, k, l \leq n$.
\end{enumerate}
\medskip

We have $VT_n = K_n \rtimes S_n$, where the conjugation action of $S_n$ on $K_n$ is given as
$$\rho \alpha_{i,j} \rho^{-1} = \alpha_{\rho(i), \rho(j)}$$
for all $1\leq i \neq j \leq n$ and $\rho\in S_n.$ 
\medskip

We begin by recalling some general results. The following three results are well-known \cite{MR1066460}.

\begin{lemma}\label{intersection of parabolic subgroups}
Let $(W, S)$ be a Coxeter system, and $X$ and $Y$ two subsets of $S$. Then $$W[X]\cap W[Y] = W[X\cap Y].$$ 
\end{lemma}

\begin{lemma}\label{amalgamated product of subgroups}
Let $(W, S)$ be a Coxeter system. Let $X$ and $Y$ be two subsets of $S$ such that $S = X \cup Y$ and the exponents $m_{s,t} = \infty$ for each $s\in X\setminus Y$ and $t\in Y\setminus X$. Then $$W = W[X]*_{W[X\cap Y]} W[Y].$$ 
\end{lemma}

A \textit{cyclic permutation} of a word  $w=x_{i_1} x_{i_2}\dots x_{i_k}$ (not necessarily reduced) is a word $w'$ (not necessarily distinct from $w$) of the form $x_{i_t}x_{i_{t+1}} x_{i_{t+2}}\dots x_{i_k} x_{i_1} x_{i_2}\cdots x_{i_{t-1}}$ for some $1\leq t \leq k$. A word is called \textit{cyclically reduced} if each of its cyclic permutation is reduced. It is immediate that a cyclically reduced word is reduced,  but the converse is not true.

\begin{lemma}\label{involutions-in-RACG}
Let $W$ be a right-angled Coxeter group and $g\in W$ a cyclically reduced word. Then $g$ is of order two if and only if $[s, t] = 1$ for every pair of generators $s$ and $t$ occurring in $g$.
\end{lemma}

The following  result on normal form for amalgamated free products is due to Serre \cite[Section 1.1, Theorem 1]{Serre}.  

\begin{lemma}\label{Normal-Form-Amalgamated-Free-Product}
Let $G_1, G_2, \ldots, G_r, H$ be a collection of groups such that $H$ is a subgroup of $G_j$ for each $1\leq j \leq r$. Consider the amalgamated free product $G = G_1 \ast_H G_2 \ast_H \cdots \ast_H G_r$. For each $1\leq j \leq r$, choose a set $T_j$ of representatives of left cosets of $H$ in $G_j$ such that $T_j$ contains the identity element 1. Then each element $g \in G$ can be written in a unique way in the form $g = t_1 t_2 \ldots t_l h$ such that:
\begin{itemize}
\item[(1)] $h \in H$ and, for each $i\in \{1, 2, \dots, l\}$, there exists $j = j(i) \in \{1, 2, \dots, r\}$ such that $t_i \in T_j \setminus \{1 \}$,
\item[(2)] $j(i)\neq j(i+1)$ for all $i\in \{1, 2, \dots, l-1\}$.
\end{itemize}
In particular, we have $g \in H$ if and only if $l = 0$ and $g = h$. 
\end{lemma}

Given a group $G$ and an automorphism $\phi$ of $G$, let $$G^\phi= \{ g \in G \mid \phi(g)=g \}$$ denotes the group of fixed-points of $\phi$. The following lemma is due to Bellingeri and Paris \cite[Lemma 3.6]{BellingeriParis2020}.

\begin{lemma}\label{Fixed-Point-Lemma}
Let $H$ be a common subgroup of groups $G_1$ and $G_2$ and $G = G_1 \ast_H G_2$ their amalgamated free product. Let $\phi : G \rightarrow G$ be an automorphism of order two such that $\phi (G_1) = G_2$ and $\phi (G_2) = G_1$. Then $G^{\phi}$ is a subgroup of $H$.
\end{lemma}

We also need the following result  \cite[Lemma 3.9]{BellingeriParis2020}.

\begin{lemma}\label{automorphism-of-amalgamated-products}
Let $H$ be a common subgroup of groups $G_1$ and $G_2$ and $G = G_1 \ast_H G_2$ their amalgamated free product. Let $\phi: G\rightarrow G$ be an automorphism of order two such that $\phi(G_1) = G_2$ and $\phi(G_2) = G_1$. Let $x \in G$ such that $\phi(x) = x^{-1}$. Then there exist $y \in G$ and $z \in H$ such that $\phi(z) = z^{-1}$ and $x = y z \phi(y)^{-1}$.
\end{lemma}

The next three subsections consisting of quite technical results occupy the rest of this section.
\medskip

\subsection{Technical results I} 
For the rest of this section, we set $$X_k = \{ \alpha_{i,j} \in \mathcal{S} \mid ~ i, j\notin \{k, k+1\} \}$$
for each $1\leq k\leq n-1$. Note that the conjugation action of $\rho_k$ is an order two automorphism of $K_n$ and its action on $K_n[X_k]$ is trivial for each  $1\leq k\leq n-1$. 

\begin{lemma}\label{relations-between-Xk}
Let $w\in \langle \rho_1, \ldots, \rho_{n-1}\rangle$ such that $w \rho_k w^{-1} = \rho_\ell$. Then $w X_k w^{-1} = X_\ell$, and consequently $w K_n[X_k] w^{-1} = K_n[X_\ell]$.
\end{lemma}

\begin{proof}
For all $1\leq t \leq n-1$, we set $\overline{X}_t = \mathcal{S} \setminus X_t$, the complement of $X_t$ in $\mathcal{S}$. Then, we have 
$$ \mathcal{S} = w \mathcal{S} w^{-1} = w ( X_k  \cup \overline{X}_k) w^{-1} =  (w X_k w^{-1})  \cup  (w \overline{X}_k w^{-1}) = (w X_k w^{-1})  \cup \overline{X}_\ell,$$
which gives $w X_k w^{-1} = \mathcal{S} \setminus \overline{X}_\ell = X_\ell$.
\end{proof}

\begin{proposition}\label{conjugation-action-of-rho1}
Let $1 \le k \le n-1$ be a fixed integer and $X$ be a subset of $\mathcal{S}$ invariant under the conjugation action of $\rho_k$. Then $$K_n[X]^{\widehat{\rho_k}} = K_n[X \cap X_k].$$
\end{proposition}

\begin{proof}
We first prove the proposition for $k=1$. 
\par
The fact that $K_n[X \cap X_1] \subseteq K_n[X]^{\widehat{\rho_1}}$ is obvious. We now prove the reverse inclusion. Let $V= \{ \alpha_{i,j} \in X ~|~ (i,j) \notin \{ (1,2), (2,1)\}\}.$ First, we prove that $K_n[X]^{\widehat{\rho_1}} \subseteq K_n[V]$. If $X = V$, then there is nothing to prove. Otherwise, since $X$ is invariant under the conjugation action of $\rho_1$, we must have $X = V \cup \{\alpha_{1,2}, \alpha_{2,1}\}. $ We set $V' = V \cup \{ \alpha_{1,2}\}$ and $V'' = V \cup \{ \alpha_{2,1}\}$. Then, by Lemma \ref{amalgamated product of subgroups}, we get $K_n[X] = K_n[V']*_{K_n[V]} K_n[V''].$
Also, $\rho_1(K_n[V'])\rho_1 = K_n[V'']$ and $\rho_1(K_n[V''])\rho_1 = K_n[V']$. Thus, by Lemma \ref{Fixed-Point-Lemma}, we get that $K_n[X]^{\widehat{\rho_1}} \subseteq K_n[V]$.
\par

More generally, for $2\leq k\leq n$ we set 
$$V_k = \{ \alpha_{i,j} \in X \mid (i, j) \notin \{1, 2\} \times \{1, 2, \ldots , k\} \}.$$
We prove by induction on $k$ that $K_n[X]^{\widehat{\rho_1}} \subseteq K_n[V_k]$. The case $k=2$ holds since $V_2=V$. Suppose that the induction hypothesis holds for $k-1$, that is $K_n[X]^{\widehat{\rho_1}} \subseteq K_n[V_{k-1}].$ Now, if $V_k=V_{k-1}$, there is nothing to prove. So we suppose that $V_k \neq V_{k-1}$. Since the set $V_{k-1}$ is invariant under the conjugation action of $\rho_1$, we have $V_{k-1} = V_k \cup \{\alpha_{1,k}, \alpha_{2,k}\}$. Set $V_k' = V_k \cup \{\alpha_{1,k}\}$ and $V_k'' = V_k \cup \{ \alpha_{2,k}\}$. By Lemma \ref{amalgamated product of subgroups}, we have $$K_n[V_{k-1}] = K_n[V_k']*_{K_n[V_k]} K_n[V_k''].$$
Also, $\rho_1(K_n[V_k'])\rho_1 = K_n[V_k'']$ and $\rho_1(K_n[V_k'']) \rho_1= K_n[V_k']$. Hence, by Lemma \ref{Fixed-Point-Lemma}, we get $K_n[X]^{\widehat{\rho_1}} \subseteq K_n[V_k]$.
\par

Next, for $2\leq k\leq n$,  we consider the set 
$$W_k = \{ \alpha_{i,j} \in X \mid (i, j) \notin \{1, 2\} \times \{1, 2, \ldots , n\} \cup \{1, 2, \ldots , k\} \times \{1, 2\}\}.$$ 
We show by induction on $k$ that $K_n[X]^{\widehat{\rho_1}} \subseteq K_n[W_k]$. The case $k=2$ holds as $W_2= V_n$. We now suppose that the induction hypothesis holds for $k-1$, that is, $K_n[X]^{\widehat{\rho_1}} \subseteq K_n[W_{k-1}]$. If $W_k=W_{k-1}$, there is nothing to prove. So, we suppose that $W_k \neq W_{k-1}$. Since the set $W_{k-1}$ is invariant under the conjugation action of $\rho_1$, we have $W_{k-1} = W_k \cup \{\alpha_{1,k}, \alpha_{2,k}\}$. Set $W_k' = W_k \cup \{\alpha_{1,k}\}$ and $W_k'' = W_k \cup \{ \alpha_{2,k}\}$. Again, by Lemma \ref{amalgamated product of subgroups}, we have $$K_n[W_{k-1}] = K_n[W_k']*_{K_n[W_k]} K_n[W_k''].$$
Also, $\rho_1(K_n[W_k'])\rho_1 = K_n[W_k'']$ and $\rho_1(K_n[W_k'']) \rho_1= K_n[W_k']$. Thus, by Lemma \ref{Fixed-Point-Lemma}, we get $K_n[X]^{\widehat{\rho_1}} \subseteq K_n[W_k]$. 
\par
Finally, we notice that $W_k=X \cap X_1$, and hence $K_n[X]^{\widehat{\rho_1}} \subseteq K_n[X \cap X_1]$. This proves the proposition for $k=1$.
\medskip

Now, we consider $k \ge 2$. Choose an element $w\in \langle \rho_1, \ldots, \rho_{n-1}\rangle$ such that  $w \rho_1 w^{-1} = \rho_k$.  Given that the set $X$ is invariant under the action of $\rho_k$. Then the set $Y= w^{-1} X w$ is invariant under the action of $\rho_1$. By earlier case, we have $K_n[Y]^{\widehat{\rho_1}} = K_n[ Y \cap X_1]$.\\
It is easy to check that $$w(K_n[Y]^{\widehat{\rho_1}}) w^{-1}= K_n[wYw^{-1}]^{\widehat{\rho_k}} = K_n[X]^{\widehat{\rho_k}}$$ and $$w (K_n[ Y \cap X_1]) w^{-1} =  K_n[ wYw^{-1} \cap X_k] = K_n[ X \cap X_k].$$
Thus, we get $K_n[X]^{\widehat{\rho_k}}=  K_n[ X \cap X_k]$, which is desired.
\end{proof}

It follows from Proposition \ref{conjugation-action-of-rho1} that
$$K_n^{\widehat{\rho_k}}  = K_n[\mathcal{S}]^{\widehat{\rho_k}}  = K_n[X_k]$$
for each $1\le k \le n-1$.

\begin{corollary}\label{centraliser sn vtn kn same}
For each $n\geq 3$, $\C_{VT_n} (S_n) = \C_{K_n} (S_n) = 1$.
\end{corollary}

\begin{proof}
Recall that $VT_n = K_n \rtimes S_n$. Let $xy \in \C_{VT_n} (S_n)$, where $x \in K_n$ and $y\in S_n$. Then, $xy\rho = \rho xy$, that is, $xy = (\rho x \rho^{-1}) (\rho y \rho^{-1})$ for each $\rho \in S_n$. This implies that $y = \rho y \rho^{-1}$ for all $\rho \in S_n$, that is, $y\in \Z(S_n) = 1$ as $n \ge 3$. Thus, we have $\C_{VT_n} (S_n) = \C_{K_n} (S_n)$. But, note that $$\C_{K_n} (S_n) \leq \bigcap_{k=1}^{n-1} \C_{K_n} (\rho_k) =\bigcap_{k=1}^{n-1} K_n^{\widehat{\rho_k}}= \bigcap_{k=1}^{n-1} K_n[X_k] = 1,$$
which is desired.
\end{proof}

Following is an analogue of \cite[Lemma 3.10]{BellingeriParis2020}.

\begin{lemma}\label{working-lemma-fork=1}
Let $X$ be a subset of $\mathcal{S}$ invariant under the conjugation action of $\rho_1$. Let $\alpha \in K_n[X]$ such that $\rho_1 \alpha \rho_1 = \alpha^{-1}$. Then there exist $\alpha'\in K_n[X]$ and $\beta \in K_n[X \cap X_1]$ such that $$\alpha = \alpha' \beta \rho_1 \alpha'^{-1} \rho_1 \quad \text{and} \quad \beta^2 = 1.$$
\end{lemma}

\begin{proof} We complete the proof in the following three steps.
\medskip

Step (1): For $2\leq k\leq n$, we set $$V_k = \{ \alpha_{i,j} \in X \mid (i, j) \notin \{1, 2\} \times \{1, 2, \ldots , k\} \}.$$ We prove by induction on $k$ that there exist $\alpha'\in K_n[X]$ and $\beta' \in K_n[V_k]$ such that $\rho_1 \beta' \rho_1 = \beta'^{-1}$ and $\alpha = \alpha' \beta' \rho_1 \alpha'^{-1} \rho_1$.
\par

Note that $V_2 = \{ \alpha_{i,j} \in X \mid (i, j) \notin \{ (1, 2), (2, 1) \}$. If $V_2 = X$, take $\alpha' = 1$, $\beta' = \alpha$, and we are done. So we assume that $V_2 \neq X$. Since $X$ is invariant under the conjugation action of $\rho_1$, we have $X = V_2 \cup \{\alpha_{1,2} , \alpha_{2,1}\}$. Set $V_2' = V_2 \cup \{\alpha_{1,2}\}$ and $V_2'' = V_2 \cup \{ \alpha_{2,1}\}$. By Lemma \ref{amalgamated product of subgroups}, we have $$K_n[X] = K_n[V_2']*_{K_n[V_2]} K_n[V_2''].$$

Also note that $\rho_1 K_n[V_2'] \rho_1 = K_n[V_2'']$ and $\rho_1 K_n[V_2''] \rho_1 = K_n[V_2']$. Recall that the conjugation action of $\rho_1$ on $K_n[X]$ is an order two automorphism of $K_n[X]$. Thus we are done for the case $k=2$ by Lemma \ref{automorphism-of-amalgamated-products}. 
\par

Suppose that $k\geq 3$ and that the induction hypothesis holds, i.e., there exist $\alpha_1'\in K_n[X]$ and $\beta_1' \in K_n[V_{k-1}]$ such that $\rho_1 \beta_1' \rho_1 = \beta_1'^{-1}$ and $\alpha = \alpha_1' \beta_1' \rho_1 \alpha_1'^{-1} \rho_1$. Now, if $V_{k} = V_{k-1}$, we are done. So, we assume that $V_{k} \neq V_{k-1}$. Since both $V_k$ and $V_{k-1}$ are invariant under the conjugation action of $\rho_1$, we have $V_{k-1} = V_k \cup \{\alpha_{1,k} , \alpha_{2,k}\}$. Set $V_k' = V_k \cup \{\alpha_{1,k}\}$ and $V_k'' = V_k \cup \{ \alpha_{2,k}\}$. By Lemma \ref{amalgamated product of subgroups}, we have $$K_n[V_{k-1}] = K_n[V_k']*_{K_n[V_k]} K_n[V_k''].$$
\par
Again the conjugation action of $\rho_1$ on $K_n[V_{k-1}]$ is an order two automorphism of $K_n[V_{k-1}]$. We also have $\rho_1 K_n[V_k'] \rho_1 = K_n[V_k'']$, $\rho_1 K_n[V_k''] \rho_1 = K_n[V_k']$ and $\rho_1 \beta_1' \rho_1 = \beta_1'^{-1}$, where $\beta_1' \in K_n[V_{k-1}]$. By Lemma \ref{automorphism-of-amalgamated-products}, there exist $\alpha_2'\in K_n[V_{k-1}]$ and $\beta' \in K_n[V_k]$ such that $\rho_1 \beta' \rho_1 = \beta'^{-1}$ and $\beta_1' = \alpha_2' \beta' \rho_1 \alpha_2'^{-1} \rho_1$. Now set $\alpha' = \alpha_1'\alpha_2'$. From induction hypothesis, we have $\alpha = \alpha_1' \beta_1' \rho_1 \alpha_1'^{-1} \rho_1$. Putting $\beta_1' = \alpha_2' \beta' \rho_1 \alpha_2'^{-1} \rho_1$, we have 
$$\alpha = \alpha_1' \alpha_2' \beta' \rho_1 \alpha_2'^{-1} \rho_1 \rho_1 \alpha_1'^{-1} \rho_1 = \alpha_1' \alpha_2' \beta' \rho_1 \alpha_2'^{-1}  \alpha_1'^{-1} \rho_1 = \alpha' \beta' \rho_1 \alpha'^{-1} \rho_1.$$
\par
We already have $\rho_1 \beta' \rho_1 = \beta'^{-1}$. This completes the proof of Step (1).
\medskip

Step (2):  For $2\leq k\leq n$, we set $$W_k = \{ \alpha_{i,j} \in X \mid (i, j) \notin \{1, 2\} \times \{1, 2, \ldots , n\} \cup \{1, 2, \ldots , k\} \times \{1, 2\}\}.$$ We now prove by induction on $k$ that there exist $\alpha'\in K_n[X]$ and $\beta' \in K_n[W_k]$ such that $\rho_1 \beta' \rho_1 = \beta'^{-1}$ and $\alpha = \alpha' \beta' \rho_1 \alpha'^{-1} \rho_1$.
\medskip

Since $W_2 = V_n$, the base case $k = 2$ is done by Step (1). Suppose that $k\geq 3$ and the induction hypothesis holds, i.e., there exist $\alpha_1'\in K_n[X]$ and $\beta_1' \in K_n[W_{k-1}]$ such that $\rho_1 \beta_1' \rho_1 = \beta_1'^{-1}$ and $\alpha = \alpha_1' \beta_1' \rho_1 \alpha_1'^{-1} \rho_1$. Now, if $W_{k} = W_{k-1}$, we are done. So, we assume that $W_{k} \neq W_{k-1}$. Since both $W_k$ and $W_{k-1}$ are invariant under the conjugation action of $\rho_1$, we have $W_{k-1} = W_k \cup \{\alpha_{k,1} , \alpha_{k,2}\}$. Set $W_k' = W_k \cup \{\alpha_{k,1}\}$ and $W_k'' =W_k \cup \{ \alpha_{k,2}\}$. By Lemma \ref{amalgamated product of subgroups}, we have $$K_n[W_{k-1}] = K_n[W_k']*_{K_n[W_k]} K_n[W_k''].$$
\par
Again the conjugation action of $\rho_1$ on $K_n[W_{k-1}]$ is an order two automorphism of $K_n[W_{k-1}]$. We also have $\rho_1 K_n[W_k'] \rho_1 = K_n[W_k'']$, $\rho_1 K_n[W_k''] \rho_1 = K_n[W_k']$ and $\rho_1 \beta_1' \rho_1 = \beta_1'^{-1}$, where $\beta_1' \in K_n[W_{k-1}]$. By Lemma \ref{automorphism-of-amalgamated-products}, there exist $\alpha_2'\in K_n[W_{k-1}]$ and $\beta' \in K_n[W_k]$ such that $\rho_1 \beta' \rho_1 = \beta'^{-1}$ and $\beta_1' = \alpha_2' \beta' \rho_1 \alpha_2'^{-1} \rho_1$. Now set $\alpha' = \alpha_1'\alpha_2'$. From induction hypothesis, we have $\alpha = \alpha_1' \beta_1' \rho_1 \alpha_1'^{-1} \rho_1$. Putting $\beta_1' = \alpha_2' \beta' \rho_1 \alpha_2'^{-1} \rho_1$, we have 
$$\alpha = \alpha_1' \alpha_2' \beta' \rho_1 \alpha_2'^{-1} \rho_1 \rho_1 \alpha_1'^{-1} \rho_1 = \alpha_1' \alpha_2' \beta' \rho_1 \alpha_2'^{-1}  \alpha_1'^{-1} \rho_1 = \alpha' \beta' \rho_1 \alpha'^{-1} \rho_1.$$
\par
We already have $\rho_1 \beta' \rho_1 = \beta'^{-1}$. This completes the proof of Step (2).
\medskip

Step (3): Note that $W_n = X \cap X_1$. Thus, we have $\alpha'\in K_n[X]$ and $\beta \in K_n[X\cap X_1]$ such that $$\alpha = \alpha' \beta \rho_1 \alpha'^{-1} \rho_1 \quad\text{and} \quad \rho_1 \beta \rho_1 = \beta^{-1}.$$ 
It only remains to be shown that $\beta^2 = 1$. But, by Proposition \ref{conjugation-action-of-rho1}, $\rho_1\beta\rho_1=\beta$. Thus, $\beta = \rho_1 \beta \rho_1 = \beta^{-1}$, and this completes the proof.
\end{proof}

Next, we generalise the preceding lemma.

\begin{lemma}\label{working-lemma}
Let $1\leq k \leq n-1$ be a fixed integer and $X$ a subset of $\mathcal{S}$ invariant under the conjugation action of $\rho_k$. Let $\alpha \in K_n[X]$ such that $\rho_k \alpha \rho_k = \alpha^{-1}$. Then there exist $\alpha'\in K_n[X]$ and $\beta \in K_n[X \cap X_k]$ such that $$\alpha = \alpha' \beta \rho_k \alpha'^{-1} \rho_k \quad \text{and} \quad\beta^2 = 1.$$
\end{lemma}

\begin{proof}
The case $k=1$ follows from Lemma \ref{working-lemma-fork=1}. So, we suppose that $k \ge 2$. Choose an element $w\in \langle \rho_1, \ldots, \rho_{n-1}\rangle$ such that $w^{-1} \rho_k w = \rho_1$, i.e., $w \rho_1 w^{-1} = \rho_k$. Note that

$$\rho_1 (w^{-1} X w) \rho_1 = w^{-1} (w \rho_1 w^{-1}) X (w \rho_1 w^{-1}) w = w^{-1} \rho_k  X  \rho_k w = w^{-1} X w,$$
i.e., $w^{-1} X w \subseteq \mathcal{S}$ is invariant under the conjugation action $\rho_1$. Similarly, we have 
$$\rho_1 (w^{-1} \alpha w) \rho_1 = w^{-1} (w \rho_1 w^{-1}) \alpha (w \rho_1 w^{-1}) w = w^{-1} \rho_k  \alpha  \rho_k w = w^{-1} \alpha^{-1} w.$$
Now the proof follows from a direct application of Lemma \ref{working-lemma-fork=1}.
\end{proof}
\medskip

\subsection{Technical results II} 
Let $\theta: VT_n  \to S_n$ and  $\lambda: S_n \rightarrow VT_n$ be as defined in Section \ref{section-prelim}.

\begin{proposition}\label{no-conjugation}
Let $1\leq k \leq n-1$ be a fixed integer and $\phi: S_n \to VT_n$ be a homomorphism such that $\theta \phi$ is identity on $S_n$. Suppose that $\phi (\tau_k) =  \beta \rho_k$ for some $\beta \in K_n[X_k]$ with $\beta^2=1$. Then $\beta = 1$.
\end{proposition}

\begin{proof}
We first prove the assertion for $k=1$. If $n\leq 3$, then $K_n[X_1] = 1$ and the assertion is vacuously true. Thus, we assume that $n\geq 4$.
\par 

Suppose that $\beta \neq 1$. We proceed to obtain a contradiction. By Lemma \ref{involutions-in-RACG}, we have $\beta = x g x^{-1}$ for some $x, g\in K_n[X_1]$ with $g\neq 1$  cyclically reduced of the form $$g = \alpha_{i_1, j_1} \alpha_{i_2, j_2} \ldots \alpha_{i_k, j_k},$$ where $3\leq i_1, j_1, i_2, j_2, \ldots, i_k, j_k\leq n$ are all distinct integers. Without loss of generality, we assume that $i_1 = 3$. Since $\theta \phi$ is identity on $S_n$, we can write $\phi (\tau_3) = x_3 \rho_3$ for some $x_3\in K_n$. As $\tau_1$ and $\tau_3$ commute with each other, we have $\beta \rho_1 x_3 \rho_3 = x_3 \rho_3 \beta \rho_1$, which can be rewritten as
\begin{equation}\label{rho1 alpha3 equation}
\beta (\rho_1 x_3 \rho_1)(\rho_3 \beta \rho_3) x_3^{-1} = 1 .
\end{equation}

Note that $\beta, \rho_3 \beta \rho_3, \rho_1 x_3 \rho_1, x_3^{-1}  \in K_n$. We define an epimorphism $\eta:K_n \to \{1, -1\}$  by setting

\begin{equation*}
\eta (\alpha_{i, j}) = 
\begin{cases}
-1 & \text{if}\ (i, j) = (3, j_1),\\
1 & \text{otherwise.}
\end{cases}
\end{equation*}
It follows from the construction of $g$ that $\eta(g)=-1$, and hence $\eta(\beta)= \eta(x)\eta(g)\eta(x)^{-1}=\eta(g)=-1$. The condition on indices $j_1,  i_2, j_2, i_3, j_3, \ldots, i_k, j_k$ imply that $\alpha_{3, j_1}$ does not appear in the expression $\rho_3 g \rho_3$, and hence $\eta (\rho_3 g  \rho_3) = 1$. Thus, we see that $$\eta (\rho_3 \beta \rho_3) = \eta (\rho_3 x g x^{-1} \rho_3) = \eta (\rho_3 x \rho_3) \eta (\rho_3 g  \rho_3)\eta (\rho_3x^{-1} \rho_3) = \eta (\rho_3 g  \rho_3)=1.$$  Since $\rho_1\alpha_{3, j_1}\rho_1= \alpha_{3, j_1}$, it follows that $\eta (\rho_1 x_3 \rho_1) = \eta (x_3)=\eta (x_3^{-1})$. Now, applying $\eta$ on \eqref{rho1 alpha3 equation} gives
$$1= \eta(\beta (\rho_1 x_3 \rho_1)(\rho_3 \beta \rho_3) x_3^{-1})=\eta (\beta)\; \eta(\rho_1 x_3 \rho_1)\; \eta(\rho_3 \beta \rho_3)\; \eta(x_3^{-1})= -1,$$
a contradiction. Hence, $\beta$ must be trivial and the proposition is proved for $k=1$. \\
 
Now, we assume that $k \ge 2$. Choose an element $w\in \langle \rho_1, \ldots, \rho_{n-1}\rangle$ such that  $w \rho_1 w^{-1} = \rho_k$. Setting $g = \theta (w)$, we get $g \tau_1 g^{-1} = \tau_k$. Consider the composition $\widehat{w}^{-1} \phi \widehat{g}: S_n \to VT_n$. It is easy to check that $\theta \widehat{w}^{-1} \phi \widehat{g}$ is identity on $S_n$. Further, note that $$\widehat{w}^{-1} \phi \widehat{g}(\tau_1) = \widehat{w}^{-1} \phi  (g \tau_1 g^{-1}) = \widehat{w}^{-1} \phi  (\tau_k) = \widehat{w}^{-1}  (\beta \rho_k) = (w^{-1} \beta w) (w^{-1} \rho_k w ) = (w^{-1} \beta w) \rho_1.$$
By Lemma \ref{relations-between-Xk}, $(w^{-1} \beta w) \in K_n[X_1]$ is an involution. Thus, by case $k=1$, we get $w^{-1} \beta w = 1$, and hence $\beta = 1$.
\end{proof}

We note that if $\phi: S_n \to VT_n$ is a homomorphism such that $\theta \phi$ is identity on $S_n$, then $\theta \widehat{w} \phi$ is also identity on $S_n$ for all $w\in K_n$, where $\widehat{w}$ is the inner automorphism of $VT_n$ induced by $w$. This together with Proposition \ref{no-conjugation}, yields the following.

\begin{corollary}\label{working-corollary-1}
Let $1\leq k \leq n-1$ be a fixed integer and $\phi: S_n \to VT_n$ be a homomorphism such that $\theta \phi$ is identity on $S_n$. Suppose that $\phi (\tau_k) =  w^{-1}\beta \rho_k w$ for some $w\in K_n$ and $\beta \in K_n[X_k]$ with $\beta^2 = 1$. Then $\beta = 1$.
\end{corollary}

\begin{corollary}\label{working-corollary-2}
Let $\phi: S_n \to VT_n$ be a homomorphism such that $\theta \phi$ is identity on $S_n$. Then, for each $1\leq k \leq n-1$, there exists $x_k\in K_n$ such that $\phi (\tau_k) = x_k \rho_k x_k^{-1}$.
\end{corollary}

\begin{proof}
Since $\theta \phi$ is identity on $S_n$, for each $1\leq k \leq n-1$, there exist $\alpha_k \in K_n$ such that $\phi (\tau_k) = \alpha_k \rho_k$. This gives $\alpha_k \rho_k \alpha_k \rho_k =1$, i.e., $ \rho_k \alpha_k \rho_k = \alpha_k^{-1}$. Now, we are done by Lemma \ref{working-lemma}, and Corollary \ref{working-corollary-1}.
\end{proof}

\begin{corollary}\label{no-conjugate}
Let $1\leq k \leq n-1$ be a fixed integer. Suppose that $\rho_k = w^{-1}\beta \rho_k w$ for some $w\in K_n$ and $\beta \in K_n[X_k]$ such that $\beta^2 = 1$. Then  $\beta = 1$.
\end{corollary}

\begin{proof}
Follows by taking $\phi =  \lambda$ in Corollary \ref{working-corollary-1}.
\end{proof}
\medskip

\subsection{Technical results III}

We say that an element $\alpha \in K_n$ satisfies condition \ref{C} if 
\begin{equation*} \label{C}
\alpha (\rho_2 \alpha^{-1} \rho_2) (\rho_2 \rho_1 \alpha \rho_1 \rho_2) (\rho_2 \rho_1 \rho_2 \alpha^{-1} \rho_2 \rho_1 \rho_2) (\rho_1 \rho_2 \alpha \rho_2 \rho_1) (\rho_1 \alpha^{-1}\rho_1) = 1, \tag{\textbf{C}}
\end{equation*}
or equivalently 
\begin{equation*}
[\alpha, \rho_2] ~\widehat{\rho_2 \rho_1}([\alpha, \rho_2])~\widehat{\rho_1}( [\rho_2 ,\alpha]) = 1,
\end{equation*}
or equivalently
\begin{equation*}
\alpha~ (\rho_2 \alpha^{-1} \rho_1 \alpha \rho_1 \rho_2)~\widehat{\rho_2\rho_1}\big((\rho_2 \alpha^{-1} \rho_1 \alpha \rho_1 \rho_2)\big)~(\rho_1 \alpha^{-1}\rho_1) = 1.
\end{equation*}

Let $\alpha = x z y$ for some $x \in K_n[X_1]$, $y \in K_n[X_2]$ and $z \in K_n$. It is easy to see that if $\alpha$ satisfies condition \ref{C}, then $z$ also satisfies condition \ref{C}.
\medskip

For the rest of this section, we assume that $n \ge4$ and fix a subset $X$ of $\mathcal{S}$ which is invariant under the conjugation action of both $\rho_1$ and $\rho_2$.  We set $U_3 = X$ and 
$$U_k = \{ \alpha_{i,j} \in X \mid (i, j)\notin \{1, 2, 3\} \times \{4, 5,\ldots, k\} \}$$  
for $4\leq k \leq n$. Note that $$U_n \subseteq U_{n-1}\subseteq \cdots \subseteq  U_k \subseteq U_{k-1} \subseteq \cdots \subseteq  U_3=X$$ and each $U_k$ is invariant under the conjugation action of both $\rho_1$ and $\rho_2$. The next three lemmas are analogues of \cite[Lemma 3.11]{BellingeriParis2020} for virtual twin groups.

\begin{lemma}\label{working-lemma-2}
Let $z\in K_n[U_{k-1}]$ satisfies condition \ref{C} for some $4\leq k \leq n$. Then there exist $x \in K_n[X_1\cap X]$, $y\in K_n[X_2\cap X]$ and $z_1\in K_n[U_k]$ such that $z = x z_1 y$ and $z_1$ satisfies condition \ref{C}.
\end{lemma}

\begin{proof}
If $z \in K_n[U_k]$, then we are done by taking $x = 1$, $y=1$ and $z_1 = z$. So, we assume that $z \in K_n[U_{k-1}]\setminus K_n[U_k]$. This implies that $U_k \neq U_{k-1}$. Since both $U_{k-1}$ and $U_k$ are invariant under the conjugation action of 
$\rho_1$ and $\rho_2$, we have $U_{k-1} = U_k \cup \{\alpha_{1, k}, \alpha_{2, k}, \alpha_{3, k}\}$. We set
\begin{eqnarray*}
G_j & :=& K_n[U_k \cup \{\alpha_{j,k}\}],  \quad 1\leq j \leq 3,\\
H & :=& G_1 \cap G_2 \cap G_3 = K_n[U_k],\\
G & :=& G_1 *_H G_2 *_H G_3 = K_n[U_{k-1}].
\end{eqnarray*}
Due to Lemma \ref{Normal-Form-Amalgamated-Free-Product}, we can write $z = a_1 a_2 \ldots a_l$ for some integer $l\geq 1$ such that 
\begin{itemize}
\item[(1)] for each $1\leq i \leq l$, there exists $j = j(i) \in \{1, 2, 3\}$ such that $a_i \in G_j \setminus H$, and
\item[(2)] $j(i) \neq j(i+1)$ for all $1\leq i \leq l-1$. 
\end{itemize}

We now argue by induction on the length $l$ of $z$. Suppose that $l=1$. Then, either $a_1 \in G_1\setminus H$ or $a_1 \in G_2\setminus H$ or $a_1 \in G_3\setminus H$. First suppose that $z = a_1 \in G_1\setminus H$. We now set
\begin{align*}
b_1 &= a_1 \in G_1\setminus H ,& b_2 &=  \rho_2 a_1^{-1} \rho_2 \in G_1 \setminus H,&
b_3 & =  \rho_2 \rho_1 a_1 \rho_1 \rho_2 \in G_3 \setminus H,&\\
b_4 & =  \rho_2 \rho_1 \rho_2 a_1^{-1} \rho_2 \rho_1 \rho_2 \in G_3 \setminus H,&
b_5 & = \rho_1 \rho_2 a_1 \rho_2 \rho_1 \in G_2 \setminus H ,& b_6 & =  \rho_1 a_1^{-1} \rho_1 \in G_2 \setminus H.&
\end{align*}

Since $z = a_1$ satisfies condition \ref{C}, we have $b_1 b_2 b_3 b_4 b_5 b_6 = 1$.  In view of Lemma \ref{Normal-Form-Amalgamated-Free-Product} and the fact that $H$ is invariant under the conjugation action of $\rho_1$ and $\rho_2$, this is possible only if $b_1 b_2$, $b_3 b_4$ and $b_5 b_6$ all lie in $H$. Set $c=a_1 \rho_2 a_1^{-1} \rho_2 \in H$.
\par

Note that 
$$\rho_2 c \rho_2 = \rho_2 a_1 \rho_2 a_1^{-1} \rho_2 \rho_2 = \rho_2 a_1 \rho_2 a_1^{-1} = c^{-1}.$$ 
Hence, by Lemma \ref{working-lemma}, there exists $\alpha \in H = K_n[U_k]$ and $\beta \in K_n [U_k \cap X_2]$ such that $c = \alpha \beta \rho_2 \alpha^{-1} \rho_2$ with $\beta^2 = 1$. This gives us $a_1 \rho_2 a_1^{-1} \rho_2 = c = \alpha \beta \rho_2 \alpha^{-1} \rho_2$, and consequently $ \rho_2 = a_1^{-1} \alpha \beta \rho_2 \alpha^{-1}a_1$, where $a_1^{-1}\alpha \in K_n$. By Corollary \ref{no-conjugate}, we have $\beta=1$. So we have $c= \alpha \rho_2 \alpha^{-1} \rho_2$. Set $y= \alpha^{-1}a_1$, $x= 1,$ and $z_1 = \alpha \in H = K_n[U_k].$ This gives us $z = a_1 = x z_1 y$ and $\rho_2 y \rho_2 = \rho_2 \alpha^{-1}a_1  \rho_2 = (\rho_2 \alpha^{-1} \rho_2) (\rho_2 a_1 \rho_2) = (\alpha^{-1}c) (c^{-1}a_1) = y$, i.e., $y \in K_n[X_2\cap X]$. So,  we are done for the case $ a_1 \in G_1 \setminus H$. 
\medskip

Now consider $a_1 \in G_2\setminus H$. We again set
\begin{align*}
b_1 &= a_1 \in G_2\setminus H ,& b_2 & = \rho_2 a_1^{-1} \rho_2 \in G_3 \setminus H,&
b_3 & = \rho_2 \rho_1 a_1 \rho_1 \rho_2 \in G_1 \setminus H,&\\
 b_4 & = \rho_2 \rho_1 \rho_2 a_1^{-1} \rho_2 \rho_1 \rho_2 \in G_2 \setminus H ,&
b_5 & = \rho_1 \rho_2 a_1 \rho_2 \rho_1 \in G_3 \setminus H ,& b_6 & = \rho_1 a_1^{-1} \rho_1 \in G_1 \setminus H.&
\end{align*}

Since $z= a_1$ satisfies condition \ref{C}, we have $b_1 b_2 b_3 b_4 b_5 b_6 = 1$. But, this leads to a contradiction, thanks to Lemma \ref{Normal-Form-Amalgamated-Free-Product}. Finally, suppose that $a_1 \in G_3\setminus H$. We again set
\begin{align*}
b_1 &= a_1 \in G_3\setminus H ,& b_2 & = \rho_2 a_1^{-1} \rho_2 \in G_2 \setminus H,&
b_3 & = \rho_2 \rho_1 a_1 \rho_1 \rho_2 \in G_2 \setminus H,&\\
b_4 & = \rho_2 \rho_1 \rho_2 a_1^{-1} \rho_2 \rho_1 \rho_2 \in G_1 \setminus H ,&
b_5 & = \rho_1 \rho_2 a_1 \rho_2 \rho_1 \in G_1 \setminus H ,& b_6 & = \rho_1 a_1^{-1} \rho_1 \in G_3 \setminus H.&
\end{align*}

Since $a_1$ satisfies condition \ref{C}, again we have $b_1 b_2 b_3 b_4 b_5 b_6 = 1$. It follows that $b_2 b_3$  and $b_4 b_5$ lie in $H$, and consequently  $a_1^{-1} \rho_1 a_1 \rho_1 = c$ (say) $ \in H$.
Note that $\rho_1 c \rho_1 = c^{-1}$. Hence, by Lemma \ref{working-lemma}, there exists $d\in H = K_n[U_k]$ and $t\in K_n [U_k \cap X_1]$ such that $c = d t \rho_1 d^{-1} \rho_1$ with $t^2 = 1$. Thus, we have $a_1^{-1} \rho_1 a_1 \rho_1 = c = d t \rho_1 d^{-1} \rho_1$, and consequently  $ \rho_1 = a_1 d t \rho_1 d^{-1}a_1^{-1}$, where $a_1 d \in K_n$. By Corollary \ref{no-conjugate}, we have $t = 1$. Now, set $x = a_1d$, $y = 1$ and $z_1 = d^{-1}\in K_n[U_k]$. It follows that $z = x z_1 y$ and $\rho_1 x \rho_1 = \rho_1  a_1d  \rho_1 = (\rho_1  a_1\rho_1) (\rho_1 d  \rho_1) = (a_1 c)(c^{-1}d) = a_1d = x$, i.e., $ x \in K_n[X_1\cap X]$. Hence, we are done for the case $l=1$.
\medskip

Let us now suppose that $l\geq 2$ and that the induction hypothesis holds. Since $z=a_1a_2 \ldots a_l$ satisfy the condition  \ref{C}, we have 
$$ (a_1\ldots a_l)\widehat{\rho_2}(a_l^{-1} \ldots a_1^{-1})\widehat{\rho_2\rho_1}(a_1\ldots a_l)\widehat{\rho_2 \rho_1 \rho_2}(a_l^{-1} \ldots a_1^{-1})\widehat{\rho_1\rho_2}(a_1\ldots a_l)\widehat{\rho_1}(a_l^{-1} \ldots a_1^{-1})=1.$$
Then, by Lemma \ref{Normal-Form-Amalgamated-Free-Product}, we have either of the following:
\begin{enumerate}
\item $a_l\widehat{\rho_2}(a_l^{-1}) \in H$,
\item $\widehat{\rho_2}(a_1^{-1})\widehat{\rho_2\rho_1}(a_1) \in H,$
\item $\widehat{\rho_2\rho_1}(a_l)\widehat{\rho_2 \rho_1 \rho_2}(a_l^{-1})=\widehat{\rho_2\rho_1}(a_l)\widehat{\rho_1 \rho_2 \rho_1}(a_l^{-1}) \in H,$
\item $\widehat{\rho_2 \rho_1 \rho_2}(a_1^{-1})\widehat{\rho_1\rho_2}(a_1)=\widehat{\rho_1 \rho_2 \rho_1}(a_1^{-1})\widehat{\rho_1\rho_2}(a_1) \in H,$
\item $\widehat{\rho_1\rho_2}(a_l)\widehat{\rho_1}(a_l^{-1}) \in H.$
\end{enumerate}
Since $H$ is invariant under the conjugation action of $\rho_1$ and $\rho_2$, we note that either of the above possibilities can be boiled down to the fact that either $a_l\widehat{\rho_2}(a_l^{-1}) \in H$ or $a_1^{-1}\widehat{\rho_1}(a_1) \in H.$

We first suppose that $c = a_l \rho_2 a_l^{-1} \rho_2  \in H$. Then, we have $\rho_2 c \rho_2 = c^{-1}.$ Repeating the above arguments, we get $c= d \rho_2 d^{-1} \rho_2$ for some $d\in H = K_n[U_k]$.
\medskip

Set $x_l = 1$, $w_l = a_{l-1}d$, $y_l = d^{-1}a_l$ and $z_l = a_1 a_2 \ldots a_{l-2} w_l.$ Then we see that $z = x_l z_l y_l$. It is easy to check that  $y_l\in K_n[X_2\cap X]$,  and hence $z_l$ satisfies condition \ref{C}. Note that the length of $z_l$ is $l-1$, whereas the length of $z$ is $l$. Thus, by induction hypothesis, there exists $x', y' \in K_n$ and $z'\in H = K_n[U_k]$ such that $z_l = x' z' y'$ with $\rho_1 x' \rho_1 = x'$ and $\rho_2 y' \rho_2 = y'$. Now set $x = x_l x'$, $y = y' y_l$ and $z = z' \in H = K_n[U_k]$. It follows that $z = x z_1 y$, $x \in K_n[X_1\cap X]$, $y\in K_n[X_2\cap X]$. Hence, we are done for the case $a_l \rho_2 a_l^{-1} \rho_2  \in H$. The case $a_l^{-1} \rho_1 a_l \rho_1  \in H$ also follows similarly, and the proof is complete. 
\end{proof}

Now, we set $V_3 = U_n$ and
$$V_k = \{ \alpha_{i,j} \in U_n \mid (i, j)\notin \{4, 5, \ldots, k\} \times \{1, 2, 3\} \}$$ 
for $4\leq k \leq n$. Note that $$V_n \subseteq V_{n-1}\subseteq \cdots \subseteq  V_k \subseteq V_{k-1} \subseteq \cdots \subseteq  V_3=U_n$$ and each $V_k$ is invariant under the conjugation action of $\rho_1$ and $\rho_2$.

\begin{lemma}\label{working-lemma-3}
Let $z\in K_n[V_{k-1}]$ such that $z$ satisfies condition \ref{C}. Then there exists $x \in K_n[X_1\cap X]$, $y\in K_n[X_2\cap X]$ and $z_1\in K_n[V_k]$ such that $z = x z_1 y$ and $z_1$ satisfies condition \ref{C}. 
\end{lemma}

\begin{proof}
We proceed by induction on $k$. The case $k=3$ holds by Lemma \ref{working-lemma-2}. We begin by noticing that if $z \in K_n[V_k]$, then we are done by taking $x = 1$, $y=1$ and $z_1 = z$. So we assume that $z \in K_n[V_{k-1}]\setminus K_n[V_k]$. Since both $V_{k-1}$ and $V_k$ are invariant under the conjugation action of 
$\rho_1$ and $\rho_2$, we have $V_{k-1} = V_k \cup \{\alpha_{k,1},\; \alpha_{k,2},\; \alpha_{k,3}\}$. We set
\begin{align*}
G_j & = K_n[V_k \cup \{\alpha_{k,j}\}],\;  1\leq j \leq 3, & H & = G_1 \cap G_2 \cap G_3 = K_n[V_k],&\\
G & = G_1 *_H G_2 *_H G_3 = K_n[V_{k-1}].&
\end{align*}

It is not difficult to check that
\begin{align*}
\rho_1 G_1 \rho_1 &= G_2,\;  \rho_1 G_2 \rho_1 = G_1,\; \rho_1 G_3 \rho_1 = G_3,&\\
\rho_2 G_1 \rho_2 &= G_1,\;  \rho_2 G_2 \rho_2 = G_3,\; \rho_2 G_3 \rho_2 = G_2.&
\end{align*}

By Lemma \ref{Normal-Form-Amalgamated-Free-Product}, we can write $z = a_1 a_2 \ldots a_l$ for some integer $l\geq 1$ such that 
\begin{itemize}
\item[(1)] for each $1\leq i \leq l$, there exists $j = j(i) \in \{1,\; 2,\; 3\}$ such that $a_i \in G_j \setminus H$, and
\item[(2)] $j(i) \neq j(i+1)$ for all $1\leq i \leq l-1$. 
\end{itemize}

As in the preceding lemma, we now argue by induction on $l$, the length of $z=a_1 a_2 \ldots a_l$. Suppose that $z=a_1$. Then, either $a_1 \in G_1\setminus H$ or $a_1 \in G_2\setminus H$ or $a_1 \in G_3\setminus H$. We first suppose that $z = a_1 \in G_1\setminus H$ and set
\begin{align*}
b_1 &= a_1 \in G_1\setminus H ,& b_2 & =  \rho_2 a_1^{-1} \rho_2 \in G_1 \setminus H,&
b_3 & =  \rho_2 \rho_1 a_1 \rho_1 \rho_2 \in G_3 \setminus H,&\\
b_4 & =  \rho_2 \rho_1 \rho_2 a_1^{-1} \rho_2 \rho_1 \rho_2 \in G_3 \setminus H,& 
b_5 & = \rho_1 \rho_2 a_1 \rho_2 \rho_1 \in G_2 \setminus H ,&
b_6 & =  \rho_1 a_1^{-1} \rho_1 \in G_2 \setminus H.&
\end{align*}

Since $z = a_1$ satisfies condition  \ref{C} we have $b_1 b_2 b_3 b_4 b_5 b_6 = 1$. In view of Lemma \ref{Normal-Form-Amalgamated-Free-Product} and the fact that $H$ is invariant under the conjugation action of $\rho_1$ and $\rho_2$, this is possible only if $b_1 b_2$, $b_3 b_4$ and $b_5 b_6$ all lie in $H$. Set $c=a_1 \rho_2 a_1^{-1} \rho_2 \in H$.
\par
Note that $$\rho_2 c \rho_2 = \rho_2 a_1 \rho_2 a_1^{-1} \rho_2 \rho_2 = \rho_2 a_1 \rho_2 a_1^{-1} = c^{-1}.$$ 
Hence, by Lemma \ref{working-lemma}, there exists $\alpha \in H = K_n[V_k]$ and $\beta \in K_n [V_k \cap X_2]$ such that $c = \alpha \beta \rho_2 \alpha^{-1} \rho_2$ with $\beta^2 = 1$. This gives us $a_1 \rho_2 a_1^{-1} \rho_2 = c = \alpha \beta \rho_2 \alpha^{-1} \rho_2$ and consequently, $$ \rho_2 = a_1^{-1} \alpha \beta \rho_2 \alpha^{-1}a_1,$$ where $\alpha^{-1}a_1 \in K_n$. By Corollary \ref{no-conjugate}, we have $\beta=1$ which means that $c= \alpha \rho_2 \alpha^{-1} \rho_2$. Now set $y= \alpha^{-1}a_1$, $x= 1,$ and $z_1 = \alpha \in H = K_n[V_k].$ This gives us $z = a_1 = x z_1 y,$ $x \in K_n[X_1\cap X]$ and $\rho_2 y \rho_2 = \rho_2 \alpha^{-1}a_1  \rho_2 = (\rho_2 \alpha^{-1} \rho_2) (\rho_2 a_1 \rho_2) = (\alpha^{-1}c) (c^{-1}a_1) = y$, i.e., $y \in K_n[X_2\cap X]$. Thus, we are done for the case $ a_1 \in G_1 \setminus H$. 
\medskip

Next suppose that $a_1 \in G_2\setminus H$. We again set
\begin{align*}
b_1 &= a_1 \in G_2\setminus H ,& b_2 & = \rho_2 a_1^{-1} \rho_2 \in G_3 \setminus H,&\\ 
b_3 & = \rho_2 \rho_1 a_1 \rho_1 \rho_2 \in G_1 \setminus H,& b_4 & = \rho_2 \rho_1 \rho_2 a_1^{-1} \rho_2 \rho_1 \rho_2 \in G_2 \setminus H ,&\\
b_5 & = \rho_1 \rho_2 a_1 \rho_2 \rho_1 \in G_3 \setminus H ,& b_6 & = \rho_1 a_1^{-1} \rho_1 \in G_1 \setminus H.&
\end{align*}

Since $z= a_1$ satisfies condition  \ref{C} we have $b_1 b_2 b_3 b_4 b_5 b_6 = 1$. This leads to a contradiction due to Lemma \ref{Normal-Form-Amalgamated-Free-Product}. Finally, suppose that $a_1 \in G_3\setminus H$. We again set
\begin{align*}
b_1 &= a_1 \in G_3\setminus H ,& b_2 & = \rho_2 a_1^{-1} \rho_2 \in G_2 \setminus H,&\\ 
b_3 & = \rho_2 \rho_1 a_1 \rho_1 \rho_2 \in G_2 \setminus H,& b_4 & = \rho_2 \rho_1 \rho_2 a_1^{-1} \rho_2 \rho_1 \rho_2 \in G_1 \setminus H ,&\\
b_5 & = \rho_1 \rho_2 a_1 \rho_2 \rho_1 \in G_1 \setminus H ,& b_6 & = \rho_1 a_1^{-1} \rho_1 \in G_3 \setminus H.&
\end{align*}

Since $a_1$ satisfies condition  \ref{C} again we have $b_1 b_2 b_3 b_4 b_5 b_6 = 1$. It follows that both $b_2 b_3$ and $b_4 b_5$ is in $H$. Hence, we have $a_1^{-1} \rho_1 a_1 \rho_1 = c$ (say) $ \in H$.
Note that $\rho_1 c \rho_1 = c^{-1}$. Hence, by Lemma \ref{working-lemma}, there exists $d\in H = K_n[V_k]$ and $t\in K_n [V_k \cap X_1]$ such that $c = d t \rho_1 d^{-1} \rho_1$ with $t^2 = 1$. Thus, we have $a_1^{-1} \rho_1 a_1 \rho_1 = c = d t \rho_1 d^{-1} \rho_1$, and consequently  $ \rho_1 = a_1 d t \rho_1 d^{-1}a_1^{-1}$, where $a_1 d \in K_n$. By Corollary \ref{no-conjugate}, we have $t = 1$. Now, set $x = a_1d$, $y = 1$ and $z_1 = d^{-1}\in K_n[V_k]$. It follows that $z = x z_1 y$ and $\rho_1 x \rho_1 = \rho_1  a_1d  \rho_1 = (\rho_1  a_1\rho_1) (\rho_1 d  \rho_1) = (a_1 c)(c^{-1}d) = a_1d = x$, i.e., $ x \in K_n[X_1\cap X]$. Hence, we are done for the case $l=1$.
\medskip

Let us now suppose that $l\geq 2$ and that the induction hypothesis holds. Since $z=a_1a_2 \ldots a_l$ satisfy the condition  \ref{C}, we have 
$$ (a_1\ldots a_l)~\widehat{\rho_2}(a_l^{-1} \ldots a_1^{-1})~\widehat{\rho_2\rho_1}(a_1\ldots a_l)~\widehat{\rho_2 \rho_1 \rho_2}(a_l^{-1} \ldots a_1^{-1})~\widehat{\rho_1\rho_2}(a_1\ldots a_l)~\widehat{\rho_1}(a_l^{-1} \ldots a_1^{-1})=1.$$
Then, by Lemma \ref{Normal-Form-Amalgamated-Free-Product}, we have either of the following:
\begin{enumerate}
\item $a_l~\widehat{\rho_2}(a_l^{-1}) \in H$,
\item $\widehat{\rho_2}(a_1^{-1})~\widehat{\rho_2\rho_1}(a_1) \in H,$
\item $\widehat{\rho_2\rho_1}(a_l)~\widehat{\rho_2 \rho_1 \rho_2}(a_l^{-1})=\widehat{\rho_2\rho_1}(a_l)~\widehat{\rho_1 \rho_2 \rho_1}(a_l^{-1}) \in H,$
\item $\widehat{\rho_2 \rho_1 \rho_2}(a_1^{-1})~\widehat{\rho_1\rho_2}(a_1)=\widehat{\rho_1 \rho_2 \rho_1}(a_1^{-1})~\widehat{\rho_1\rho_2}(a_1) \in H,$
\item $\widehat{\rho_1\rho_2}(a_l)~\widehat{\rho_1}(a_l^{-1}) \in H.$
\end{enumerate}
Since $H$ is invariant under the conjugation action of $\rho_1$ and $\rho_2$, either of the above possibilities can be reduced to the fact that either $a_l ~\widehat{\rho_2}(a_l^{-1}) \in H$ or $a_1^{-1}~\widehat{\rho_1}(a_1) \in H.$
\medskip

We first suppose that $c = a_l \rho_2 a_l^{-1} \rho_2  \in H$. So, we get $\rho_2 c \rho_2 = c^{-1}.$ Repeating the above arguments, we get $c= d \rho_2 d^{-1} \rho_2$, for some $d\in H = K_n[V_k]$.
Set $x_l = 1$, $w_l = a_{l-1}d$, $y_l = d^{-1}a_l$ and $z_l = a_1 a_2 \ldots a_{l-2} w_l.$ Note that $z = x_l z_l y_l$. It is easy to check that $y_l\in K_n[X_2\cap X]$, and hence $z_l$ satisfies condition \ref{C}. Note that the length of $z_l$ is $l-1$, whereas the length of $z$ is $l$. Hence by induction hypothesis, there exists $x', y' \in K_n$ and $z'\in H = K_n[V_k]$ such that $z_l = x' z' y'$ with $\rho_1 x' \rho_1 = x'$, $\rho_2 y' \rho_2 = y'$. Now set $x = x_l x'$, $y = y' y_l$ and $z = z' \in H = K_n[V_k]$. It follows that $z = x z_1 y$, $x \in K_n[X_1\cap X]$, $y\in K_n[X_2\cap X]$. Hence, we are done for the case $a_l \rho_2 a_l^{-1} \rho_2  \in H$. The other case follow in a similar manner.
\end{proof}

Now, we set
\begin{equation}\label{W11}
W_{1,1} = \{ \alpha_{i, j}\in X \mid (i ,j) \in \{ (1, 2), (2, 3), (3, 1)\} \}
\end{equation}
and
\begin{equation}\label{W12}
W_{1, 2} = \{ \alpha_{i, j} \in X \mid (i, j) \in \{ (2, 1), (3, 2), (1, 3) \} \}.
\end{equation}
It is easy to check that $\rho_1(W_{1,1})\rho_1 = \rho_2(W_{1,1})\rho_2 = W_{1, 2}$ and $\rho_1(W_{1, 2})\rho_1 = \rho_2(W_{1, 2})\rho_2 = W_{1,1}$.

\begin{lemma}\label{working-lemma-4}
If $z\in K_n[W_{1,1}]\ast K_n[W_{1,2}]$ satisfies condition \ref{C}, then $z = 1$.
\end{lemma}

\begin{proof}
Suppose that $z\neq 1$. Set $G_1 = K_n[W_{1,1}]$, $G_2 = K_n[W_{1,2}]$ and $G = G_1 \ast G_2$. Due to Lemma \ref{Normal-Form-Amalgamated-Free-Product}, we can write $z = a_1 a_2 \ldots a_l$ for some integer $l\geq 1$ such that 
\begin{itemize}
\item[(1)] for each $1\leq i \leq l$, there exists $j = j(i) \in \{1, 2 \}$ such that $a_i \in G_j \setminus \{1\}$, and
\item[(2)] $j(i) \neq j(i+1)$ for all $1\leq i \leq l-1$. 
\end{itemize}
\medskip

We now argue by induction on $l$ to arrive at a contradiction. Suppose that $l=1$. Then either $a_1 \in G_1\setminus \{1\}$ or $a_1 \in G_2\setminus \{1\}$. We can assume that $a_1 \in G_1\setminus \{1\}$, since the case $a_1 \in G_2\setminus \{1\}$ goes along parallel lines. We now set
\begin{align*}
b_1 &= a_1 \in G_1\setminus \{1\},& b_2 & =  \rho_2 a_1^{-1} \rho_2 \in G_2 \setminus \{1\},&\\ 
b_3 & =  \rho_2 \rho_1 a_1 \rho_1 \rho_2 \in G_1 \setminus \{1\},& b_4 & =  \rho_2 \rho_1 \rho_2 a_1^{-1} \rho_2 \rho_1 \rho_2 \in G_2 \setminus \{1\},&\\
b_5 & = \rho_1 \rho_2 a_1 \rho_2 \rho_1 \in G_1 \setminus \{1\},& b_6 & =  \rho_1 a_1^{-1} \rho_1 \in G_2 \setminus \{1\}.&
\end{align*}

Since $z = a_1$ satisfies condition \ref{C}, we have $b_1 b_2 b_3 b_4 b_5 b_6 = 1$. But, this leads to a contradiction due to Lemma \ref{Normal-Form-Amalgamated-Free-Product}. 
\par
Now suppose that $l\geq 2$ and induction hypothesis holds. Due to  Lemma \ref{Normal-Form-Amalgamated-Free-Product}, we have either 
$a_l^{-1} \rho_1 a_l \rho_1  = 1$ or $a_l \rho_2 a_l^{-1} \rho_2 =1$.
We suppose that $a_l \rho_2 a_l^{-1} \rho_2 = 1$, since the case $a_l^{-1} \rho_1 a_l \rho_1  = 1$ also goes in parallel lines. Thus, we have $\rho_2 a_l \rho_2 = a_l.$ Set $x = 1$, $w = a_1 a_2 \ldots a_{l-1}\in G$, and $y = a_l$. Note that $z = x w y$ and $w\neq 1$ (since length of $w$ is non-zero). Since  $\rho_2 a_l \rho_2 = a_l$, it follows that $y=a_l \in K_n[X_2]$, and hence $w$ satisfies condition \ref{C}. Note that the length of $w$ is $l-1$, whereas the length of $z$ is $l$. Hence by induction hypothesis, $w = 1$, a contradiction.
\end{proof}

\begin{lemma}\label{working-lemma-5}
Let $n\geq 3$ and $X$ a subset of $\mathcal{S}$ that is invariant under the conjugation action of $\rho_1$ and $\rho_2$. Suppose $\phi: S_n \rightarrow VT_n$ is a homomorphism such that $\theta \phi$ is identity on $S_n$ with $\phi(\tau_1) = \rho_1$ and $\phi(\tau_2) = \alpha \rho_2 \alpha^{-1}$ for some $\alpha \in K_n[X]$. Then there exists $\alpha_1 \in K_n[X_1]$ and $\alpha_2\in K_n[X_2]$ such that $\alpha = \alpha_1 \alpha_2$.
\end{lemma}

\begin{proof}
Since $(\tau_2 \tau_1)^3 = 1$, we have $(\alpha \rho_2 \alpha^{-1} \rho_1)^3 =1$, which upon expansion shows that $\alpha$ satisfies condition \ref{C}. We proceed along the following three steps.
\medskip

Step (1): We first show by induction on $3\leq k \leq n$ that there exists $x \in K_n[X_1\cap X]$, $y\in K_n[X_2\cap X]$, and $z\in K_n[U_k]$ such that $\alpha = x z y$ and $z$ satisfies condition \ref{C}. The case $k=3$ is trivial, since we can take $x = y = 1$ and $z= \alpha$. Suppose that $4\leq k \leq n$ and induction hypothesis holds, i.e., there exists $x_1 \in K_n[X_1\cap X]$, $y_1 \in K_n[X_2\cap X]$ and $z_1\in K_n[U_{k-1}]$ such that $\alpha = x_1 z_1 y_1$ and $z_1$ satisfies condition \ref{C}. By Lemma \ref{working-lemma-2}, we get $x_2 \in K_n[X_1\cap X]$, $y_2\in K_n[X_2\cap X]$, and $z_2\in K_n[U_k]$ such that $z_2$ satisfies condition \ref{C} and $z_1 = x_2 z_2 y_2$. Thus, we have $\alpha = x_1 x_2 z_2 y_2 y_1$. Taking $x = x_1x_2$, $z=z_2$ and $y=y_2y_1$ completes Step (1).
\medskip

Step (2): We now show by induction on $3\leq k \leq n$ that there exists $x \in K_n[X_1\cap X]$, $y\in K_n[X_2\cap X]$, and $z\in K_n[V_k]$ such that $\alpha = x z y$ and $z$ satisfies condition \ref{C}. Since $V_3 = U_n$, we are done for the case $k =3$ from Step (1). Now suppose that $4\leq k \leq n$ and induction hypothesis holds, i.e., there exists $x_1 \in K_n[X_1\cap X]$, $y_1 \in K_n[X_2\cap X]$, $z_1\in K_n[V_{k-1}]$ such that $\alpha = x_1 z_1 y_1$ and $z_1$ satisfies condition \ref{C}. By Lemma \ref{working-lemma-3}, we get $x_2 \in K_n[X_1\cap X]$, $y_2\in K_n[X_2\cap X]$, and $z_2\in K_n[V_k]$ such that $z_2$ satisfies condition \ref{C} and $z_1 = x_2 z_2 y_2$. Thus, we have $\alpha = x_1 x_2 z_2 y_2 y_1$. Taking $x = x_1x_2$, $z=z_2$ and $y=y_2y_1$ completes Step (2).
\medskip

Step (3): We now set
\begin{align*}
W_1 & = \{ \alpha_{i,j} \in X \mid (i, j)\in \{1, 2, 3\} \times \{1, 2, 3\} \} ~\text{and}&\\
W_2 &= \{ \alpha_{i,j} \in X \mid (i, j)\in \{4, 5, \ldots, k\} \times  \{4, 5, \ldots,k\} \}.&
\end{align*}

It follows that $V_n = W_1 \sqcup W_2$ and hence
$$K_n[V_n] = K_n[W_1] \times K_n[W_2].$$ 
Also note that 
\begin{align*}
K_n[W_1] & = K_n[W_{1,1}]\ast K_n[W_{1,2}] ~\text{and} & \\
 K_n[W_2] & = K_n[X_1\cap X] \cap K_n[X_2\cap X],&
\end{align*}
where $W_{1,1}$ and $W_{1,2}$ are given by \eqref{W11} and \eqref{W12}. By Step (2), we have $x \in K_n[X_1\cap X]$, $y\in K_n[X_2\cap X]$ and $z\in K_n[V_n] = K_n[W_1] \times K_n[W_2]$ such that $\alpha = xzy$, and $z$ satisfies condition \ref{C}. Let $z_1\in K_n[W_1]$ and $z_2\in K_n[W_2]$ such that $z = z_1 z_2 = 1z_1 z_2$. It is easy to check that that $z_2\in K_n[X_2]$, and hence $z_1$ satisfies condition \ref{C}. Now, Lemma \ref{working-lemma-4} gives $z_1 = 1$. Thus, we get $\alpha = xz_2y$, where $x \in K_n[X_1\cap X]$, $y\in K_n[X_2\cap X]$ and $z_2\in K_n[W_2] = K_n[X_1\cap X] \cap K_n[X_2\cap X].$ Taking $\alpha_1 = x$ and $\alpha_2 = z_2y$ completes the proof of the lemma.
\end{proof}

Lemma \ref{working-lemma-5} can be generalised as follows.

\begin{lemma}\label{working-lemma-6}
Let $n\geq 3$ and $X$ a subset of $\mathcal{S}$ that is invariant under the conjugation action of both $\rho_k$ and $\rho_{k+1}$ for some fixed $1\leq k \leq n-2$. Suppose that $\phi: S_n \rightarrow VT_n$ is a homomorphism such that $\theta \phi$ is identity on $S_n$, $\phi(\tau_k) = \rho_k$ and $\phi(\tau_{k+1}) = \alpha \rho_{k+1} \alpha^{-1}$ for some $\alpha \in K_n[X]$. Then there exists $\alpha_1 \in K_n[X_k\cap X]$ and $\alpha_2\in K_n[X_{k+1}\cap X]$ such that $\alpha = \alpha_1 \alpha_2$.
\end{lemma}

\begin{proof}
The case $k=1$ is considered in Lemma \ref{working-lemma-5}. So, we assume that $k \ge 2$. Choose an element $w\in \langle \rho_1, \ldots, \rho_{n-1}\rangle$ such that $w^{-1} \rho_k w = \rho_1$ and $w^{-1} \rho_{k+1} w = \rho_2$. It is not difficult to see that such $w$ exists. In fact, if $w_0$ is the cycle $(1,2, \dots, n)$, then taking $w=w_0^{k-1}$, we see that $w^{-1} \rho_k w=\rho_{1}$ and $w^{-1} \rho_{k+1} w=\rho_2$. Set $g = \theta (w)$. It follows that $g^{-1} \tau_k g = \tau_1$, and $g^{-1} \tau_{k+1} g = \tau_2$.
\par
Now set $\overline{\phi} = \widehat{w}^{-1} \phi \widehat{g}$, where $\widehat{w}$ and $\widehat{g}$ are inner automorphisms induced by $w$ and $g$ in $VT_n$ and $S_n$, respectively. Notice that 
\begin{align*}
\theta \widehat{w}^{-1} \phi \widehat{g} (\tau) & = \theta \widehat{w}^{-1} \phi  (g \tau g^{-1}) & \\
& = \theta \widehat{w}^{-1} (\phi (g) \phi (\tau) \phi (g^{-1}))&\\
& = \theta (w^{-1} \phi (g) \phi (\tau) \phi (g^{-1})w) &\\
& = \theta (w^{-1}) \theta( \phi (g)) \theta(\phi (\tau)) \theta(\phi (g^{-1})) \theta(w) &\\
&= g^{-1} g \tau g^{-1} g \\
&= \tau&
\end{align*}
for all $\tau \in S_n$, and hence $\theta \widehat{w}^{-1} \phi \widehat{g} (\tau)$ is identity on $S_n$. Also, note that 
\begin{align*}
\widehat{w}^{-1} \phi \widehat{g} (\tau_1) &= \widehat{w}^{-1} \phi  (g \tau_1 g^{-1}) = \widehat{w}^{-1} \phi  (\tau_k) = \widehat{w}^{-1}  (\rho_k) = w^{-1}  \rho_k w = \rho_1 \quad  \text{and} &\\
\widehat{w} \phi \widehat{g}^{-1} (\tau_2) &= \widehat{w} \phi  (g \tau_2 g^{-1}) = \widehat{w} \phi  (\tau_{k+1}) = \widehat{w}  (\alpha \rho_{k+1} \alpha^{-1}) = (w^{-1} \alpha w) \rho_2 (w^{-1}  \alpha^{-1} w).&
\end{align*}
The result now follows from Lemma \ref{working-lemma-5} and Lemma \ref{relations-between-Xk}.
\end{proof}
\medskip

\subsection{Main results} We now prove the main results of this section. For each $1\leq k\leq n$, we set $$Y_k = \{\alpha_{i, j} \in \mathcal{S} \mid k \leq i \neq j \leq n\}.$$
We note that $$\emptyset=Y_n \subseteq Y_{n-1}\subseteq \cdots \subseteq  Y_k \subseteq Y_{k-1} \subseteq \cdots \subseteq  Y_1=\mathcal{S}.$$

\medskip

\begin{proposition}\label{phi-conjugate-to-iota}
Let $n\geq 3$ and $\phi: S_n \rightarrow VT_n$ a homomorphism such that $\theta \phi$ is identity on $S_n$. Then $\phi$ is conjugate to $\lambda$.
\end{proposition}

\begin{proof}
We claim that for each $1\leq m \leq n-1$, there exists an inner automorphism $\widehat{w}$ of $VT_n$ induced by some $w\in K_n$ such that $\widehat{w} \phi (\tau_i) = \rho_i$ for all $1\leq i \leq m$. The case $m = n-1$ would establish the proposition.
\medskip

We prove the claim by induction on $m$. By Corollary \ref{working-corollary-2}, there exists $x_1 \in K_n$ such that $\phi (\tau_1) = x_1 \rho_1 x_1^{-1}$. This gives us $\widehat{x_1}^{-1} \phi (\tau_1) = \rho_1$. We now consider the case $m=2$. Set $\phi_1 = \widehat{x_1}^{-1} \phi$, and note that $\theta \phi_1$ is also identity on $S_n$. Again, by Corollary \ref{working-corollary-2}, there exists $x_2 \in K_n$ such that $\phi_1 (\tau_2) = x_2 \rho_2 x_2^{-1}$. We also have $\phi_1 (\tau_1) = \rho_1$. Thus, by Lemma \ref{working-lemma-5}, there exists $\alpha_1 \in K_n[X_1]$, $\alpha_2\in K_n[X_2]$ such that $x_2 = \alpha_1 \alpha_2$. This gives $\phi_1 (\tau_2) = x_2 \rho_2 x_2^{-1} = \alpha_1 \alpha_2 \rho_2 \alpha_2^{-1}\alpha_1^{-1} = \alpha_1 \rho_2 \alpha_1^{-1}$. We can check that  $\widehat{x_1 \alpha_1}^{-1} \phi(\tau_i)=\rho_i$ for $i=1, 2$, and the claim holds for $m=2$.
\medskip

Now, suppose that $3\leq m \leq n-1$ and that the claim holds for $m-1$, i.e., there exists some $w\in K_n$ such that $\widehat{w} \phi (\tau_i) = \rho_i$ for all $1\leq i \leq m-1$.  Set $\phi_{m-1} = \widehat{w} \phi$.  Since $\theta \phi_{m-1}$ is identity on $S_n$, there exists $y\in K_n$ such that $\phi_{m-1} (\tau_m) = y \rho_m$. Since $\tau_m \tau_i = \tau_i \tau_m$, we have $y \rho_m \rho_i = \rho_i y \rho_m$ for all $1\leq i \leq m-2$. Thus, $\rho_i y \rho_i =  y$, i.e., $y$ is a fixed-point under the conjugation action of  $\rho_i$ for each $1\leq i \leq m-2$. By Corollary \ref{conjugation-action-of-rho1}, we have $y\in K_n^{\widehat{\rho_i}} = K_n[X_i]$ for each $1\leq i \leq m-2$. Since 
$$\bigcap_{i=1}^{m-2} X_i = Y_m,$$
it follows that $y \in K_n[Y_m]$. Since $\tau_m^2 = 1$, we have $1 = \phi_{m-1} (\tau_m)^2 = y \rho_m y \rho_m$, i.e., $ \rho_m y \rho_m = y^{-1}$. Note that $Y_m$ is invariant under conjugation by $\rho_m$. Thus, by Lemma \ref{working-lemma}, there exist $x_m \in K_n[Y_m] \leq K_n[Y_{m-1}]$ and $\beta \in K_n[Y_m\cap X_m] \leq K_n[X_m]$ such that $y = x_m \beta \rho_m x_m^{-1} \rho_m$ and $\beta^2 =1$. This gives
$$\phi_{m-1} (\tau_m) = y \rho_m = x_m \beta \rho_m x_m^{-1}.$$
Now, due to Corollary \ref{working-corollary-1},  we have $\beta = 1$, and consequently $\phi_{m-1} (\tau_m) = x_m \rho_m x_m^{-1}$.
\medskip

Note that $Y_{m-1}$ is invariant under the conjugation action of both $\rho_{m-1}$ and $\rho_m$. Due to Lemma \ref{working-lemma-6}, there exists $u\in K_n[Y_{m-1}\cap X_{m-1}] = K_n[Y_{m+1}]$ and $v\in K_n[Y_{m-1}\cap X_m]$ such that $x_m = uv$. Thus, we have $\phi_{m-1} (\tau_m) = x_m \rho_m x_m^{-1} = u \rho_m u^{-1}$. It follows from the choice of $u$ that  $\widehat{u}^{-1} \phi_{m-1}(\tau_i)=\rho_i$ for all $1\leq i \leq m$. This is equivalent to $\widehat{uw}^{-1} \phi(\tau_i)=\rho_i$ for all $1\leq i \leq m$, and the proof is complete.
\end{proof}

Finally, we present the main result of this section. 

\begin{theorem}\label{Homomorphisms-Sn-VTn}
Let $n, m$ be integers such that $n \geq m$, $n \geq 5$ and $m \geq 2$. Let $\phi : S_n \to VT_m$ be a homomorphism. Then, upto conjugation of homomorphisms, one of the following  assertions holds:
\begin{enumerate}
\item $\phi$ is abelian,
\item $n=m$ and $\phi = \lambda$, 
\item $n=m=6$ and $\phi = \lambda~\nu$.
\end{enumerate}
\end{theorem}

\begin{proof}
Consider the composition $S_n \stackrel{\phi}{\longrightarrow}  VT_m  \stackrel{\theta}{\longrightarrow}  S_m.$ By Proposition \ref{Homorphisms-Sn-Sm}, one of the following  assertions holds for $\theta \phi$: 
\begin{enumerate}
\item $\theta \phi$ is abelian,
\item $n=m$ and $\theta \phi = \id$,
\item $n=m=6$ and $\theta \phi = \nu$.
\end{enumerate}
\medskip
Case (1): Let $\theta \phi$ be abelian. We claim that there exist $w \in S_m$ such that $\theta \phi(\tau_i)=w$ for all $i$. Suppose on the contrary that there exist $i$ and elements $w_1 \neq w_2$  in $S_m$ such that $\theta \phi(\tau_i)=w_1$ and $\theta \phi(\tau_{i+1})=w_2.$ The braid relation $\tau_i\tau_{i+1}\tau_i = \tau_{i+1}\tau_i\tau_{i+1}$ gives $w_1w_2w_1= w_2w_1w_2$. Now, $\theta \phi$ being abelian implies that $w_1 = w_2$, and the claim holds.
\par
Let us set $\lambda(w)=g$. It follows that, for each $1 \le i\le n-1$, there exist $\alpha_i \in K_n$ such that $\phi(\tau_i)=\alpha_i g$. 
We claim that $\alpha_i=\alpha_{i+1}$ for all $i$. Since $\tau_i^2=1$, it follows that $g^2=1$ and $\alpha_i g = g \alpha_i^{-1}$ for all $i$. Since $(\tau_i\tau_{i+1})^3=1$, it follows that $1=(\alpha_i g \alpha_{i+1} g)^3=(\alpha_i g g \alpha_{i+1}^{-1})^3=(\alpha_i \alpha_{i+1}^{-1})^3$. Thus, the element $\alpha_i \alpha_{i+1}^{-1}\in K_n$ has order dividing three. Since $K_n$ is a right-angled Coxeter group, a non-trivial finite order element must have order two \cite[Proposition 1.2]{Ruiz-Kazachkov-Remeslennikov}. This implies that $\alpha_i=\alpha_{i+1}$. Hence, the homomorphism $\phi$ is abelian.
\medskip

Case(2): Let $n=m$ and $\theta \phi = \id$. Then, by Proposition \ref{phi-conjugate-to-iota}, the homomorphism $\phi$ is conjugate to $\lambda$.
\medskip

Case(3): Lastly, suppose that $n=m=6$ and $\theta \phi = \nu$. Then $\theta \phi \nu^{-1}$ is identity on $S_n$. By Proposition \ref{phi-conjugate-to-iota}, the homomorphism $ \phi \nu^{-1}$ is conjugate to $\lambda$, and equivalently, $\phi$ is conjugate to $ \lambda \nu$.
\end{proof}
\medskip

\section{Homomorphisms from $VT_n$ to $VT_m$}\label{Homomorphisms-VTn-VTm}
Recall that the non-inner automorphism $\nu : S_6 \to S_6$ is defined on generators as
$$ \nu(\tau_1) = (1,2)(3,4)(5,6), \quad \nu(\tau_2) = (2,3)(1,5)(4,6), \quad \nu(\tau_3) = (1,3)(2,4)(5,6),$$
$$ \nu(\tau_4) = (1,2)(3,5)(4,6), \quad \nu(\tau_5) = (2,3)(1,4)(5,6).$$
We set $v_i= \lambda\nu(\tau_i)$ for each $1 \leq i \leq 5$.

\begin{lemma}\label{For-Proof-of-Case-n=m=6}
If $H= \langle v_3, v_4, v_5 \rangle $, then $KT_6^H := \{ x \in KT_6 ~|~ w x w^{-1} =x \text{ for all } w \in H \} = \{1\}$. 
\end{lemma}

\begin{proof}
Consider the subset $H' = \{ v_3, v_4, v_5, v_3 v_4 v_3, v_4 v_5 v_4, v_3 v_4 v_5 v_4 v_3 \} $ of $H$. Since $H$ is a subgroup of $\langle \rho_1, \ldots, \rho_5 \rangle \cong S_6$, we  
can view elements of $H'$ in terms of permutations as
$$ v_3 = (1,3)(2,4)(5,6), \quad v_4 = (1,2)(3,5)(4,6 ),\quad u_5 = (2,3)(1,4)(5,6),$$
$$ v_3 v_4 v_3 = (1,6)(2,5)(3,4), \quad  v_4 v_5 v_4 = (1,5)(2,6)(3,4), \quad v_3 v_4 v_5 v_4 v_3 = (1,2)(3,6)(4,5) .$$
\medskip

For fixed $1 \leq i < j \leq 6$, set $U_{ij}= \mathcal{S} \setminus \{\alpha_{i,j}, \alpha_{j,i}\}$, $U'_{ij} = U_{ij} \cup \{ \alpha_{i,j} \}$ and $U''_{ij} = U_{ij} \cup \{ \alpha_{j,i} \}$. Then, by Lemma \ref{amalgamated product of subgroups}, we have 
$$KT_6 = KT_6[U'_{ij}] { \ast_{KT_6[ U_{ij}]}} KT_6[U''_{ij}].$$
Note that the set $H'$ is taken in a way that each transposition $(i,j)$ appears in the decomposition of some element of $H'$. Let us choose an element $w \in H'$ containing the transposition $(i, j)$ in its decomposition. We notice that $w(KT_6[U'_{ij}])w= KT_6[U''_{ij}]$ and $w(K_6[U''_{ij}])w= K_6[U'_{ij}]$. Thus, by Lemma \ref{Fixed-Point-Lemma}, we have $KT_6^H \subseteq {KT_6}^{\widehat{w}} \subseteq KT_6[U_{ij}]$ for all $1 \leq  i < j \leq 6$. Since $\bigcap_{1 \leq  i < j \leq 6} U_{ij} = \emptyset$, by Lemma \ref{intersection of parabolic subgroups}, we have $KT_6^H= \{1\}.$
\end{proof}

For each $m \in \mathbb{Z}$, consider the homomorphism $\phi_m : VT_n \to VT_n$  given on generators by 
$$\phi_m (s_i) = (s_i \rho_i)^{m} \rho_i \quad \textrm{and} \quad \phi_m(\rho_i) = \rho_i.$$

Setting $\phi_{-1}=\zeta$, the main result of this section is as follows.
 
\begin{theorem}\label{homomorphisms-VTn-VTm}
Let $n, m$ be integers such that $n \geq m$, $n \geq 5$ and $m \geq 2$. Let $\phi : VT_n \to VT_m$ be a homomorphism. Then, upto conjugation of homomorphisms, one of the following  assertions holds:
\begin{enumerate}
\item $\phi$ is abelian,
\item $n=m$ and $\phi \in \{\lambda  \pi,~ \lambda  \theta, ~\phi_m, ~ \zeta  \phi_m, ~\text{where}~ m \in \mathbb{Z} \}$,
\item $n=m=6$ and $\phi \in \{ \lambda  \nu  \theta, \lambda  \nu  \pi \}$.
\end{enumerate}
\end{theorem}

\begin{proof}
Consider the composition $S_n \stackrel{\lambda}{\longrightarrow}  VT_n \stackrel{\phi}{\longrightarrow}  VT_m.$ By Proposition \ref{Homomorphisms-Sn-VTn}, one of the following  assertions holds for $\phi  \lambda$: 
\begin{enumerate}
\item $\phi  \lambda$ is abelian,
\item $n=m$ and $\phi  \lambda = \lambda$,
\item $n=m=6$ and $\phi  \lambda = \lambda  \nu$.
\end{enumerate}
\medskip
Case(1): This case is similar to Case (1) of Theorem \ref{Homomorphisms-VTn-Sm}. If $\phi  \lambda$ is abelian, then there exists $w \in VT_m$ such that $\phi  \lambda(\tau_i) = w$ for all $i$. Equivalently, $\phi(\rho_i)=w$ for all $i$. Let $\phi(s_1)=z$. The relation $\rho_i s_{i+1} \rho_i = \rho_{i+1} s_i \rho_{i+1}$ gives $\phi(s_i) = \phi(s_{i+1})=z$ for all $i$. Finally, the relation $s_1 \rho_3 = \rho_3 s_1$ gives $zw=wz$, and hence $\phi$ is abelian. 
\medskip

Case(2): Let $m=n$ and $\phi  \lambda = \lambda$. This implies that $\phi(\rho_i) = \rho_i$ for all $i$. We now determine $\phi(s_i)$ for all $i$. Since $VT_n = K_n \rtimes S_n$, we have $\phi(s_i)=a_i \lambda(w_i)$ for some $a_i \in K_n$ and $w_i \in S_n$. For each $ 3 \leq j \leq n-1$, we have
$$\tau_j w_1 = \theta  \phi(\rho_j s_1) = \theta  \phi(s_1\rho_j)= w_1 \tau_j.$$
This implies that $w_1$ lies in centraliser of $\langle \tau_3, \tau_4, \dots, \tau_{n-1} \rangle$ in $S_n$, which is $\langle \tau_1 \rangle$. Thus, either $w_1 =1$ or $w_1= \tau_1$.
\medskip

Case (2a): Let us suppose that $w_1=1$. For each $3 \leq k \leq n-1$, we have $s_1 \rho_k = \rho_k s_1$, and hence $a_1  = \rho_k a_1\rho_k$. Thus, if we set $X_k= \{ \alpha_{i, j} \in \mathcal{S} \mid i, j \not\in \{k, k+1 \} \}$, then $a_1 \in K_n[X_k]$ for all $3 \le k \le n-1$. We have 
$$\bigcap_{3 \leq k \leq n-1} X_k =\{ \alpha_{1,2}, \alpha_{2,1} \},$$
and hence $a_1 \in K_n[\{\alpha_{1,2}, \alpha_{2,1}\}] \cong \mathbb{Z}_2 \ast \mathbb{Z}_2$. Elements in $K_n[\{\alpha_{1,2}, \alpha_{2,1}\}]$ are of the form $(\alpha_{1,2} \alpha_{2,1})^m$ or $\alpha_{1,2} (\alpha_{2,1} \alpha_{1,2})^m$ or $\alpha_{2,1} (\alpha_{1,2} \alpha_{2,1})^m$ for some integer $m$. The only order two elements are $\alpha_{1,2}$, $\alpha_{2,1}$, $\alpha_{1,2} (\alpha_{2,1} \alpha_{1,2})^m$ and $\alpha_{2,1} (\alpha_{1,2} \alpha_{2,1})^m$. Since $a_1^2=1$, it follows that  $a_1$ is either $1$ or any of the order two element mentioned beforehand. We use the relation $s_{i+1}= \rho_i \rho_{i+1} s_i \rho_{i+1} \rho_i$ to determine $a_i$ as follows.
\begin{itemize}
\item If $a_1=1$, then $\phi(s_i)=1$ for all $i$. Thus, we obtain $\phi = \lambda  \theta.$
\item If $\phi(s_1)= \alpha_{1,2}= s_1$, then $\phi(s_i) = \alpha_{i,i+1}= s_i$  for all $i$. Consequently, we have $\phi = \id$.
\item If $a_1= \alpha_{2,1} = \rho_1 s_1 \rho_1$, then $\phi(s_i)= \alpha_{i+1,i} = \rho_i s_i \rho_i$ for all $i$. Thus, we have $\phi = \zeta$.
\item Let $\phi(s_1) = a_1 = \alpha_{1,2} (\alpha_{2,1} \alpha_{1,2})^m = s_1( \rho_1 s_1 \rho_1 s_1)^m = s_1 (\rho_1 s_1)^{2m} = (s_1 \rho_1)^{2m} s_1= (s_1 \rho_1)^{2m+1} \rho_1 $.  Then we get $\phi(s_i) = (s_i \rho_i)^{2m+1} \rho_i$ for all $i$, and hence $\phi= \phi_{2m+1}$.
\item Lastly, let $\phi(s_1) = a_1 = \alpha_{2,1} (\alpha_{1,2} \alpha_{2,1})^m = (\rho_1 s_1 \rho_1) ( s_1 \rho_1 s_1 \rho_1)^m = \rho_1 (s_1 \rho_1)^{2m+1} $.  Then we have 
$\phi(s_i) = \rho_i  (s_i \rho_i)^{2m+1} $ for all $i$, and consequently $\phi= \zeta  \phi_{2m+1}$.
\end{itemize}
\par

Case (2b): Suppose that $w_1=\tau_1$. Then $\phi(s_1) = a_1 \rho_1$, and hence $\rho_1 a_1 \rho_1 = a_1^{-1}$. As in Case (2a), the commuting relation $s_1 \rho_k = \rho_k s_1$  for  $3 \leq k \leq n-1$ shows that $a_1\in K_n[\{ \alpha_{1,2}, \alpha_{2,1}\}]$. A direct check shows that an elements of $K_n[\{ \alpha_{1,2}, \alpha_{2,1}\}]$ satisfying $\rho_1 a_1 \rho_1 = a_1^{-1}$ must be of the form $(\alpha_{1,2} \alpha_{2,1})^m$, where $m \in \mathbb{Z}$. 
\begin{itemize}
\item If $a_1= 1$, then $\phi(s_i)=\rho_i$ for all $i$, and hence $\phi = \lambda  \pi$. 
\item Let $a_1=(\alpha_{1,2} \alpha_{2,1})^m$ for some non-zero integer $m$. Then $\phi(s_1)= (s_1 \rho_1 s_1 \rho_1)^m \rho_1= (s_1 \rho_1)^{2m} \rho_1$. Using the relation $s_2= \rho_1 \rho_2 s_1 \rho_2 \rho_1$ gives $$\phi(s_2)= \rho_1 \rho_2 (\alpha_{1,2} \alpha_{2,1})^m \rho_1 \rho_2 \rho_1= \rho_1 \rho_2( \alpha_{1,2} \alpha_{2,1})^m \rho_2 \rho_1 \rho_2= (\alpha_{2,3} \alpha_{3,2})^m \rho_2 = (s_2 \rho_2)^{2m} \rho_2.$$ Iterating the process gives $\phi(s_i)= (s_i \rho_i)^{2m} \rho_i$ for all $i$. Thus,  $\phi = \phi_{2m}$ for some $m \in \mathbb{Z}$.
\end{itemize}
\medskip
Case(3): Let $n=m=6$ and $\phi  \lambda = \lambda  \nu$. Let $u_i = \nu(\tau_i)$ and $v_i = \lambda (u_i)$ for all $1 \le i \le 5$.  Then we have $\phi(\rho_i) = v_i$ for all $i$. We set $\phi(s_i) = a_i \lambda(w_i)$, where $a_i \in K_6$ and $w_i \in S_6$. We note that for $i=3, 4, 5$, we have $u_i w_1 = (\theta  \phi)(s_1 \rho_i) =  (\theta  \phi)(\rho_i s_1) = w_1 u_i$. This implies that $w_1$ belongs to the centraliser of $\langle u_3, u_4, u_5 \rangle$ in $S_6$, and hence $w_1 \in \langle u_1 \rangle$. Further, we notice that $\lambda(w_1) v_i = v_i \lambda(w_1)$ for $i=3, 4, 5$. The relation $s_1 \rho_i = \rho_i s_1$ gives $a_1 = v_i a_1 v_i^{-1}$ for $i=3, 4, 5$. It follows from Lemma \ref{For-Proof-of-Case-n=m=6} that $a_1=1$, and hence $\phi(s_1) = \lambda(w_1) \in \langle v_1 \rangle$. We now have two possibilities: either $\phi(s_1)=1$ or $\phi(s_1)= v_1$. If $\phi(s_1)=1$, then $\phi(s_i)=1$ for all $i$, and hence $\phi = \lambda  \nu  \theta$. If $\phi(s_1)= v_1$, then $\phi(s_i)= v_i$ for all $i$, and hence $\phi = \lambda  \nu  \pi$. This proves the theorem.
\end{proof}

We now build the set-up for determining $\Aut(VT_n)$.

\begin{proposition}\label{phim is not auto}
The following statements hold for each $n \geq 2$:
\begin{enumerate}
\item $\phi_m$ is not surjective for each even $m$.
\item $\phi_{m}$ is an automorphism of $VT_n$ for $m=1, -1$.
\item $\phi_m$ is injective but not surjective for each odd $m \neq 1, -1$.
\end{enumerate}
\end{proposition}

\begin{proof}
Assertion (1) is immediate since the induced map on the abelianisation is not surjective for each even $m$. Assertion (2) is also clear since $\phi_1$ is the identity automorphism and $\phi_{-1}=\zeta$ is an order two automorphism of $VT_n$.
\medskip

For assertion (3), we first observe that $\phi_{-m} = \zeta \phi_m$ for each integer $m$. For, 
\begin{eqnarray*}
\zeta (\phi_m (s_i)) & =& \zeta((s_i \rho_i)^m \rho_i) = (\rho_i s_i \rho_i \rho_i)^m \rho_i = (\rho_is_i)^m \rho_i = (s_i \rho_i)^{-m} \rho_i \\
& =& \phi_{-m} (s_i).
\end{eqnarray*}
Thus, it is enough to take $m=2t+1$, where $t \ge 1$. It is easy to see that 
 $$\phi_m(\alpha_{i,i+1})= \alpha_{i,i+1} (\alpha_{i+1, i} \alpha_{i, i+1})^t \quad \textrm{and} \quad
\phi_m(\alpha_{i,j})= \alpha_{i,j} (\alpha_{j, i} \alpha_{i, j})^t$$
for all $1 \leq i \neq j \leq n$. This implies that $\phi_m(K_n) \subseteq K_n$. We observe that, if $w = \alpha_{i_1,j_1}\cdots \alpha_{i_k,j_k}$ is a reduced word of length $k$, then $\phi_m(w) = \alpha_{i_1,j_1} (\alpha_{j_1, i_1} \alpha_{i_1, j_1})^t \cdots \alpha_{i_k,j_k} (\alpha_{j_k, i_k} \alpha_{i_k, j_k})^t$ is a reduced word of length $(2t+1)k$.
\par
Let $xy\in VT_n$ with $x\in K_n$ and $y\in S_n$ such that  $1 = \phi_m(xy) = \phi_m(x)y$. It follows that $y=1$ and $\phi_m(x)=1$. The preceding observation implies that $x=1$, and hence $\phi_m$ is injective. Let $xy\in VT_n$ with $x\in K_n$ and $y\in S_n$ such that  $\alpha_{1,2} = \phi_m(xy) = \phi_m(x)y$. It follows that $y=1$ and $\phi_m(x)=\alpha_{1,2}$. Comparing lengths of  $x$ and $\phi_m(x)=\alpha_{1,2}$, the preceding observation leads to a contradiction. Hence, $\phi_m$ is not surjective, which proves assertion (3). 
\end{proof}

Recall that a group is called {\it co-Hopfian} if every injective endomorphism is surjective. Proposition \ref{phim is not auto}(3) yields the following result.

\begin{corollary}\label{vtn not cohopfian}
The groups $VT_n$ and $KT_n$ are not co-Hopfian for each $n \ge2$. 
\end{corollary}

\begin{theorem}\label{aut out vtn}
For $n \ge 5$, $\Aut(VT_n) = \Inn(VT_n) \rtimes \langle \zeta \rangle \cong VT_n \rtimes \mathbb{Z}_2$ and $\Out(VT_n) \cong \langle \zeta \rangle \cong \mathbb{Z}_2$.
\end{theorem}

\begin{proof}
We first claim that the map $\zeta$ is a non-inner automorphism of $VT_n$ for each $n \ge 3$. Suppose that $\zeta$ is an inner automorphism of $VT_n$, say, $\zeta = \widehat{z}$ for some $z \in VT_n$. Since $\widehat{z}(\rho_i)=\zeta(\rho_i)=\rho_i$ for all $i$, we have $z \in \C_{VT_n}(S_n)$. By Corollary \ref{centraliser sn vtn kn same}, we have $z=1$. But, this is a contradiction since $\zeta$ is not the identity map.
\par
It follows from Theorem \ref{homomorphisms-VTn-VTm}(2) and Proposition \ref{phim is not auto} that any automorphism $\phi$ of $VT_n$ is of the form $\phi= \widehat{w} \zeta$ for some $w \in VT_n$. We already showed above that $\zeta$ is a non-inner automorphism. Further, it is known from \cite[Corollary 4.2]{NaikNandaSingh2020} that $\Z(VT_n)=1$. Hence, $\Aut(VT_n) = \Inn(VT_n) \rtimes \langle \zeta \rangle \cong VT_n \rtimes \mathbb{Z}_2$ and $\Out(VT_n) = \langle \zeta \rangle \cong \mathbb{Z}_2$.
\end{proof}

Note that $VT_2 \cong T_3 \cong \mathbb{Z}_2 * \mathbb{Z}_2$, and $\Aut(VT_2) \cong \Inn(VT_2) \rtimes \mathbb{Z}_2$ by \cite[Theorem 6.1(1)]{NaikNandaSingh1}. The groups $VT_3$ and $VT_4$ need to be dealt with separately with the latter appearing to be more challenging. We leave these cases for the interested readers.
\medskip

We conclude by tabulating the status of some properties of braid groups, virtual braid groups, twin groups, virtual twin groups and their pure subgroups. 
\medskip

\begin{small}
\begin{center}\label{table1}
\begin{tabular}{|c||c|c|c|c|}
  \hline
 $G$ & $B_n$ & $P_n$  & $VB_n$ & $VP_n$ \\
&Braid group& Pure braid group& Virtual braid group& Pure virtual braid group\\
&&&&\\
   \hline
      \hline
\textrm{Hopfian} & Yes for $n \ge2$ & Yes for $n \ge2$ & Yes for $n \ge5$ & Yes for $n=2$\\
&By linearity& By linearity&\cite{BellingeriParis2020}& $VP_2$ is free group of rank two\\
&&&&Unknown for $n \ge3$\\
  \hline
\textrm{Co-Hopfian}  & No for $n \ge2$ & No for $n \ge2$ & Yes for $n \ge5$ & No for $n=2$\\
&Easy to see \cite{MR2253663}& Easy to see\cite{MR2253663}& \cite{BellingeriParis2020}&Unknown for $n \ge3$\\
 &&&&\\
  \hline
   \textrm{$\Aut(G)$}  & Known for $n \ge2$& Known for $n \ge2$& Known  for $n \ge 5$&Known for $n =2$\\
&\cite{DyerGrossman}& \cite{BellMargalit} & \cite{BellingeriParis2020} &Unknown for $n \ge3$\\
&&&&\\
  \hline
\end{tabular}\\
\medskip
Table 1
\end{center}
\end{small}
\medskip

\begin{small}
\begin{center}\label{table2}
\begin{tabular}{|c||c|c|c|c|}
  \hline
 $G$ & $T_n$ & $PT_n$  & $VT_n$ & $PVT_n$\\
&Twin group& Pure twin group& Virtual twin group& Pure virtual twin group\\
&&&&\\
   \hline
     \hline
\textrm{Hopfian} & Yes for $n \ge2$ &  Yes for $n \ge2$ & Yes for $n \ge2$ &  Yes for $n \ge2$\\
&By linearity&By linearity&\cite{NaikNandaSingh2020}& $PVT_n$ is a RAAG \cite{NaikNandaSingh2020}\\
&&&&\\
  \hline
 \textrm{Co-Hopfian} & No for  $n \ge 3$ & No for $3 \le n \le 6$ & No for $n \ge2$ & No for $n\ge 2$\\
& \cite{NaikNandaSingh2} & $PT_n$ is free for $3 \le n \le 5$& This work&\cite{NaikNandaSingh2020}\\
&& and $PT_6$ is RAAG \cite{BarVesSin, GonGutiRoq, MostRoq}&&\\
&&Unknown for $n \ge 7$&&\\
&&&&\\
  \hline 
   \textrm{$\Aut(G)$}    & Known for $n \ge2$& Known for  $2 \le n \le 5$& Known for $n \ge 5$ & Known for $n \ge2$\\
& \cite{NaikNandaSingh1}& Unknown for $n \ge 6$&This work& \cite{NaikNandaSingh2020}\\
&&&&\\
  \hline
\end{tabular}\\
\medskip
Table 2
\end{center}
\end{small}
\medskip

\begin{ack}
Tushar Kanta Naik would like to acknowledge support from NBHM via grant 0204/3/2020/R\&D-II/2475. Neha Nanda thanks IISER Bhopal for the Institute Post-Doctoral Fellowship. Mahender Singh is supported by the Swarna Jayanti Fellowship grants DST/SJF/MSA-02/2018-19 and SB/SJF/2019-20/04.
\end{ack}


\medskip


\begin{thebibliography}{1}
 
\bibitem{Artin} Emil Artin, \textit{ Braids and permutations}, Ann. Math. 48 (1947), 643--649.
 
\bibitem{BarVesSin} Valeriy Bardakov, Mahender Singh and Andrei Vesnin, \textit{Structural aspects of twin and pure twin groups}, Geom. Dedicata 203 (2019), 135--154.


\bibitem{BartholomewFennKamada2019} Andrew Bartholomew, Roger Fenn, Naoko Kamada and Seiichi  Kamada, \textit{Colorings and doubled colorings of virtual doodles}, Topology  Appl. 264 (2019),  290--299.

\bibitem{BartholomewFennKamada2018} Andrew Bartholomew, Roger Fenn, Naoko Kamada and Seiichi  Kamada, \textit{Doodles on surfaces}, J. Knot Theory Ramifications 27 (2018), no. 12, 1850071, 26 pp.

\bibitem{BartholomewFennKamada2018-2} Andrew Bartholomew, Roger Fenn, Naoko Kamada and Seiichi  Kamada, \textit{On Gauss codes of virtual doodles}, J. Knot Theory Ramifications 27 (2018), no. 11, 1843013, 26 pp.

\bibitem{MR2253663}  Robert W. Bell and Dan Margalit, \textit{Braid groups and the co-Hopfian property}, J. Algebra 303 (2006), no. 1, 275--294. 

\bibitem{BellMargalit} Robert W. Bell and Dan  Margalit, \textit{Injections of Artin groups}, Comment. Math. Helv. 82 (2007), 725--751.

\bibitem{BellingeriParis2020}  Paolo Bellingeri and Luis Paris, \textit{Virtual braids and permutations}, Ann. Inst. Fourier (Grenoble) 70 (2020), no. 3, 1341--1362.


\bibitem{Bourbaki}  Nicolas Bourbaki, \textit{Lie Groups and Lie Algebras. Chapters 4--6}, Translated from the 1968 French Original by Andrew Pressley. Elements of Mathematics (Berlin), xii+300 pp. Springer, Berlin (2002).

\bibitem{Ruiz-Kazachkov-Remeslennikov} Montserrat Casals-Ruiz, Ilya Kazachkov and Vladimir Remeslennikov, \textit{Elementary equivalence of right-angled Coxeter groups and graph products of finite abelian groups}, Bull. Lond. Math. Soc. 42 (2010), no. 1, 130--136.

\bibitem{CisnerosFloresJuyumayaMarquez} Bruno Cisneros, Marcelo Flores, Jes\'us Juyumaya and Christopher Roque-M\'{a}rquez,  \textit{An Alexander type invariant for doodles},  J. Knot Theory Ramifications  31 (2022), 2250090.


\bibitem{Farley2021} Daniel S. Farley, \textit{The planar pure braid group is a diagram group}, (2021), arXiv:2109.02815.

\bibitem{DyerGrossman} Joan L. Dyer and Edna K.  Grossman, \textit{The automorphism groups of the braid groups}, Amer. J. Math. 103 (1981), 1151--1169.


\bibitem{FennTaylor} Roger Fenn and Paul Taylor, \textit{Introducing doodles}, Topology of low-dimensional manifolds (Proc. Second Sussex Conf., Chelwood Gate, 1977), pp. 37--43, Lecture Notes in Math., 722, Springer, Berlin, 1979.


\bibitem{MR1410467} Roger Fenn, Richard Rimanyi and Colin Rourke, \textit{The braid-permutation group}, Topology 36 (1997), no. 1, 123--135. 


\bibitem{Gaudreau2020} Robin Gaudreau, \textit{The braid group injects in the virtual braid group}, arXiv:2008.09631.


\bibitem{GonGutiRoq} Jes\'us Gonz\'alez,  Jos\'{e} Luis Le\'on-Medina and Christopher  Roque-M\'arquez, \textit{Linear motion planning with controlled collisions and pure planar braids}, Homology Homotopy Appl. 23 (2021), no. 1, 275--296.

\bibitem{Gotin} Konstantin Gotin, \textit{Markov theorem for doodles on two-sphere}, (2018), arXiv:1807.05337.


\bibitem{HarshmanKnapp}  Nathan L. Harshman and Adam C. Knapp, \textit{Anyons from three-body hard-core interactions in one dimension}, Ann. Physics 412 (2020), 168003, 18 pp.

\bibitem{HarshmanKnapp2021}  Nathan L. Harshman and Adam C. Knapp, \textit{Topological exchange statistics in one dimension}, Phys. Rev. A 105 (2022), no. 5, Paper No. 052214, 15 pp.

 
 \bibitem{MR1066460} James E. Humphreys, \textit{Reflection groups and Coxeter groups}, Cambridge Studies in Advanced Mathematics, 29. Cambridge University Press, Cambridge, 1990. xii+204 pp.

 
\bibitem{Kamada} Seiichi Kamada, \textit{ Invariants of virtual braids and a remark on left stabilizations and virtual exchange moves}, Kobe J. Math. 21 (2004), no. 1-2, 33--49.

\bibitem{Khovanov} Mikhail Khovanov, \textit{Doodle groups}, Trans. Amer. Math. Soc. 349 (1997), 2297--2315.


\bibitem{MR1997331} Greg Kuperberg, \textit{What is a virtual link?}, Algebr. Geom. Topol. 3 (2003), 587--591.


\bibitem{Lin-1} Vladimir Ja. Lin, \textit{ Artin braids and related groups and spaces}, Algebra. Topology. Geometry, Vol. 17 (Russian), pp. 159--227, 308, Akad. Nauk SSSR, Vsesoyuz. Inst. Nauchn. i Tekhn. Informatsii, Moscow, 1979.

\bibitem{Lin-2} Vladimir Ja. Lin, \textit{ Representation of a braid group by permutations}, Uspehi Mat. Nauk 27 (1972), no. 3 (165), 192.

\bibitem{Magnus1966}   Wilhelm Magnus, Abraham Karrass and  Donald  Solitar, \textit{Combinatorial group theory, Presentations of groups in terms of generators and relations}, Interscience Publishers, New York-London-Sydney 1966 xii + 444 pp.

\bibitem{Mostovoy} Jacob Mostovoy, \textit{A presentation for the planar pure braid group}, (2020), arXiv:2006.08007.

\bibitem{MostRoq} Jacob Mostovoy and Christopher Roque-M\'arquez, \textit{Planar pure braids on six strands}, J. Knot Theory Ramifications  29  (2020), No. 01, 1950097.


\bibitem{NaikNandaSingh1} Tushar Kanta Naik, Neha Nanda and Mahender Singh, \textit{Conjugacy classes and automorphisms of twin groups},  Forum Math. 32 (2020), no. 5, 1095--1108.

\bibitem{NaikNandaSingh2} Tushar Kanta Naik, Neha Nanda and Mahender Singh, \textit{Some remarks on twin groups},  J. Knot Theory Ramifications 29 (2020), no. 10, 2042006, 14 pp.

\bibitem{NaikNandaSingh2020} Tushar Kanta Naik, Neha Nanda and Mahender Singh, \textit{Structure and automorphisms of pure virtual twin groups}, Monatsh. Math.  202 (2023), 555--582.

\bibitem{NandaSingh2020} Neha Nanda and Mahender Singh,  \textit{Alexander and Markov theorems for virtual doodles}, New York J. Math. 27 (2021), 272--295.


\bibitem{Serre} Jean-Pierre Serre, \textit{Arbres, amalgames, $SL_2$}, With an English summary. Written with the collaboration of Hyman Bass. Aste'risque, No. 46. Mathematical Society of France, Paris, 1977. 189 pp.

\bibitem{ShabatVoevodsky}  George B. Shabat and Vladimir Voevodsky, \textit{Drawing curves over number fields}, The Grothendieck Festschrift, Vol. III, 199--227, Progr. Math., 88, Birkh\"{a}user Boston, Boston, MA, 1990.

\end{thebibliography}
\end{document}